\documentclass[12pt]{amsart}

\usepackage[english]{babel}
\usepackage[T1]{fontenc}

\usepackage[a4paper]{geometry}
\usepackage{amsmath}
\usepackage{verbatim}
\usepackage[colorinlistoftodos]{todonotes}
\usepackage{mathtools} 

\usepackage[all]{xy}
\usepackage{amssymb,mathrsfs, amscd, graphicx, array, multirow,
enumerate, amsfonts, euscript}
\usepackage{comment}
\usepackage{upgreek} 

\usepackage{marvosym}
\xyoption{poly}
\usepackage{url}

\allowdisplaybreaks
\usepackage{color}
\usepackage{xcolor}

\usepackage{tikz-cd}
\usepackage{tikz}
\usetikzlibrary{matrix,arrows}
\tikzset{cong/.style={draw=none,edge node={node [sloped, allow upside down, auto=false]{$\cong$}}},
         Isom/.style={above,every to/.append style={edge node={node [sloped, allow upside down, auto=false]{$\sim$}}}}}

\theoremstyle{plain}
\newtheorem{thm}{Theorem}[section]
\newtheorem{lemma}[thm]{Lemma}
\newtheorem{lem}[thm]{Lemma}
\newtheorem{prop}[thm]{Proposition}
\newtheorem{cor}[thm]{Corollary}

\theoremstyle{definition}

\newtheorem{defn}[thm]{Definition}

\newtheorem{exg}[thm]{Example}

\theoremstyle{remark}
\newtheorem{remark}[thm]{Remark}
\newtheorem{rmk}[thm]{Remark}
\newtheorem{rem}[thm]{Remark}

\numberwithin{equation}{section}

\def\Sel{{\rm Sel}}

\def\makeop#1{\expandafter\def\csname#1\endcsname
  {\mathop{\rm #1}\nolimits}\ignorespaces}
\makeop{Hom}   \makeop{End}   \makeop{Aut}   \makeop{Isom}  \makeop{Pic} 
\makeop{Gal}   \makeop{ord}   \makeop{Char}  \makeop{Div}   \makeop{Lie} 
\makeop{PGL}   \makeop{Corr}  \makeop{PSL}   \makeop{sgn}   \makeop{Spf}
\makeop{Spec}  \makeop{Tr}    \makeop{Nr}    \makeop{Fr}    \makeop{disc}
\makeop{Proj}  \makeop{supp}  \makeop{ker}   \makeop{im}    \makeop{dom}
\makeop{coker} \makeop{Stab}  \makeop{SO}    \makeop{SL}    \makeop{SL}
\makeop{Cl}    \makeop{cond}  \makeop{Br}    \makeop{inv}   \makeop{rank}
\makeop{id}    \makeop{Fil}   \makeop{Frac}  \makeop{GL}    \makeop{SU}
\makeop{Nrd}   \makeop{Sp}    \makeop{Tr}    \makeop{Trd}   \makeop{diag}
\makeop{Res}   \makeop{ind}   \makeop{depth} \makeop{Tr}    \makeop{st}
\makeop{Ad}    \makeop{Int}   \makeop{tr}    \makeop{Sym}   \makeop{can}
\makeop{length}\makeop{SO}    \makeop{torsion} \makeop{GSp} \makeop{Ker}
\makeop{Adm}   \makeop{Mat}   \makeop{Cor}   \makeop{Inf} \makeop{Cok}
\makeop{Ver}
\def\makebb#1{\expandafter\def
  \csname bb#1\endcsname{{\mathbb{#1}}}\ignorespaces}
\def\makebf#1{\expandafter\def\csname bf#1\endcsname{{\bf
      #1}}\ignorespaces} 
\def\makegr#1{\expandafter\def
  \csname gr#1\endcsname{{\mathfrak{#1}}}\ignorespaces}
\def\makescr#1{\expandafter\def
  \csname scr#1\endcsname{{\EuScript{#1}}}\ignorespaces}
\def\makecal#1{\expandafter\def\csname cal#1\endcsname{{\mathcal
      #1}}\ignorespaces} 

\def\doLetters#1{#1A #1B #1C #1D #1E #1F #1G #1H #1I #1J #1K #1L #1M
                 #1N #1O #1P #1Q #1R #1S #1T #1U #1V #1W #1X #1Y #1Z}
\def\doletters#1{#1a #1b #1c #1d #1e #1f #1g #1h #1i #1j #1k #1l #1m
                 #1n #1o #1p #1q #1r #1s #1t #1u #1v #1w #1x #1y #1z}
\doLetters\makebb   \doLetters\makecal  \doLetters\makebf
\doLetters\makescr 
\doletters\makebf   \doLetters\makegr   \doletters\makegr
     \def\qed{\qedmark\medbreak}%
\def\qedmark{{\enspace\vrule height 6pt width 5pt depth 1.5pt}}%
    
\def\Gm{{{\bbG}_{\rm m}}}

\def\Spec{{\rm Spec}\,}

\def\wh{\widehat}
\def\wt{\widetilde}
\def\Ind{{\rm Ind}}

\def\Gm{\mathbb{G}_{\rm m}} 
\def\G{\mathbb{G}}
\def\R{\mathbb{R}}
\def\Q{\mathbb{Q}}
\def\Z{\mathbb{Z}}
\def\A{\mathbb{A}}

\def\embed{\hookrightarrow}

\def\pr{{\rm pr}}

\newcommand{\npr}{\noindent }

\newcommand{\isoto}{\stackrel{\sim}{\to}}

\makeatletter
\newcommand{\xdashrightarrow}[2][]
  {\ext@arrow 0359\rightarrowfill@@{#1}{#2}}
\newcommand{\xdashleftarrow}[2][]
  {\ext@arrow 3095\leftarrowfill@@{#1}{#2}}
\newcommand{\xdashleftrightarrow}[2][]{\ext@arrow 3359\leftrightarrowfill@@{#1}{#2}}
\def\rightarrowfill@@{\arrowfill@@\relax\relbar\rightarrow}
\def\leftarrowfill@@{\arrowfill@@\leftarrow\relbar\relax}
\def\leftrightarrowfill@@{\arrowfill@@\leftarrow\relbar\rightarrow}
\def\arrowfill@@#1#2#3#4{%
  $\m@th\thickmuskip0mu\medmuskip\thickmuskip\thinmuskip\thickmuskip
   \relax#4#1
   \xleaders\hbox{$#4#2$}\hfill
   #3$%
}
\makeatother

\usepackage[OT2,T1]{fontenc}
\DeclareSymbolFont{cyrletters}{OT2}{wncyr}{m}{n}
\DeclareMathSymbol{\Sha}{\mathalpha}{cyrletters}{"58}

\def\Gmk{\G_{{\rm m}, k}}

\def\vol{{\rm vol}}

\DeclareMathOperator{\ld}{ld}

\usepackage{color}

\begin{document}

\title{Cohomological properties of multinorm-one tori}

\author{Pei-Xin Liang}
\address{(Liang) Institute of Mathematics, Academia Sinica, 6F Astronomy Mathematics Building, No.~1, Roosevelt Rd. Sec.~4  
Taipei, Taiwan, 10617}
\email{cindy11420@gmail.com}

\author{Yasuhiro Oki}
\address{(Oki) Department of Mathematics, Faculty of Science, Hokkaido University, 060-0810, Sapporo, Hokkaido, Japan}
\email{oki@math.sci.hokudai.ac.jp}

\author{Chia-Fu Yu}
\address{(Yu) Institute of Mathematics, Academia Sinica and NCTS,
6F Astronomy Mathematics Building, No.~1, Roosevelt Rd. Sec.~4  
Taipei, Taiwan, 10617}
\email{chiafu@math.sinica.edu.tw}

\date{\today}

\maketitle

\begin{abstract}
  We investigate the Tate--Shafarevich group $\Sha^1(k, T)$ of a multinorm-one torus $T$ over a global field $k$. 
  We explore a few functorial maps among cohomology groups and their relations. Using these properties and the approach 
  in Bayer-Fluckiger--Lee--Parimala  [Adv.~in Math., 2019], we obtain a similar basic structural result for $\Sha^1(k, T)$ and extends a few results in loc.~cit.~for some more general multinorm-one tori.  
  We also give an alternative proof of previous results on the vanishing of $\Sha^1(k, T)$ due to Demarche--Wei and the case of a product of two abelian extensions due to Pollio. 
  Moreover, we improve the explicit description of $\Sha^1(k, T)$ of Lee by removing an intersection condition.  
\end{abstract}

\section{Introduction}\label{sec:I}

Let $k$ be a global field and let $L$ be a finite \'etale algebra over $k$, that is, $L=\prod_{i=0}^m K_i$ is a product of some finite separable field extensions $K_i$ of $k$. 
Let $T_{L/k}$ be the multinorm-one torus $T_{L/k}$ over $k$ associated to $L/k$; see Section~\ref{sec:TLk}. The group of its $k$-valued points is the kernel of the norm map
\[ N_{L/k}\colon L^\times=\prod_{i=0}^m K_i^\times \to k^\times ; \quad (x_i)\mapsto \prod_{i=0}^m N_{K_i/k}(x_i). \]

One of the main interests is to know whether the Hasse norm principle holds for $L/k$. That is, if an element $x\in k^\times$ is a local norm everywhere for $N_{L/k}$, whether or not it lies in $N_{L/k}(L^\times)$. The special case where $L=K$ is a field extension of $k$ is the original problem of Hasse. 
Let 
\begin{equation}\label{eq:I.1}
   \Sha(L/k):=\frac{k^\times \cap N_{L/k} (\A_L^\times)}{N_{L/k} (L^\times)} 
\end{equation}
be the Tate--Shafarevich group for $L/k$, where $\A_L=\prod_{i} \A_{K_i}$ is the ad\`ele ring of $L$. This is the obstruction group for the Hasse norm principle. Studying the group $\Sha(L/k)$
is the same as that studying the Tate--Shafarevich group $\Sha^1(k, T_{L/k})$  of the torus $T_{L/k}$ (see \eqref{eq:ShaT} and \eqref{eq:2.4}), and it has been investigated by several authors; see \cite{hurlimann, prasad-rapinchuk, pollio-rapinchuk,pollio,demarche-wei,BLP19, lee:jpaa2022} and the references therein.
An explicit
formula for the class number of $T_{L/k}$ is given by Fan-Yun Hung et al.~\cite{hung-yu}.

In \cite{BLP19} the authors gave an approach for computing $\Sha^1(k,T_{L/k})$ or equivalently its Pontryagin dual $\Sha^1(k,T_{L/k})^\vee \simeq \Sha^2(k, \wh T_{L/k})$ by Poitou--Tate duality. Here $\wh T$ denotes the character group of a $k$-torus $T$. 
Write $L=K\times K'$ as the product of two \'etale $k$-algebras.
Set $E:=K\otimes_k K'$ and define a morphism of $k$-tori 
\begin{equation}\label{eq:f}
   f \colon R_{E/k}(\mathbb{G}_{m,E})\longrightarrow R_{L/k}(\mathbb{G}_{m,L}), \quad f(x)=(N_{E/K}(x)^{-1},N_{E/K'}(x)), 
\end{equation}
where $R_{L/k}$ denotes the Weil restriction of scalars from $L$ to $k$.
Set $S_{K,K'}$ to be the kernel of $f$; it is a $k$-torus and there are natural isomorphisms (see \cite[Lemma 3.1]{BLP19} and Equation \eqref{eq:2.4}) 
\begin{equation}\label{eq:I.2}
 \Sha^1(k, \wh{S}_{K,K'})\simeq \Sha^2(k, \wh{T}_{L/k})\simeq \Sha(L/k)^\vee.   
\end{equation}
When $K/k$ is a cyclic extension, the structure and properties of $\Sha^1(k, \wh{S}_{K,K'})$ are essentially understood. For example, when each component $K_i$ of $L$
is cyclic, computing $\Sha(L/k)$ can be reduced to computing $\Sha(L(p)/k)$ of the maximal $k$-subalgebra $L(p)$ of $p$-power degree, and some special cases have been explicitly computed; see Section 8 of loc.~cit. Moreover, T.-Y.~Lee gave an explicit description of the Tate--Shafarevich group $\Sha(L(p)/k)$; see \cite[Theorem 6.5]{lee:jpaa2022}. \\

In this paper we continue to study $\Sha^1(k,\wh{S}_{K,K'})$, based on the earlier work mentioned above, for more general \'etale $k$-algebras $L$. 
Inspired by the strategy 
in \cite{BLP19}, we construct two finite abelian groups $S(K,K')$ and $D(K,K')$ and relate the quotient group $S(K,K')/D(K,K')$ with $\Sha^1(k, \wh{S}_{K,K'})$.
From \cite[Section 2]{BLP19}, there is a short exact sequence of $k$-tori:
\begin{equation}
\begin{tikzcd}
1\arrow[r] & S_{K,K'} \arrow[r] & R_{K'/k}(T_{E/K'}) \arrow[r,"{N_{E/K}}"] &T_{K/k} \arrow[r] & 1.  
\end{tikzcd}
\end{equation}
Taking the dual we obtain a short exact 
sequence of $\Gamma_k$-modules:
\begin{center}
\begin{tikzcd}
0 \arrow[r] &\wh{T}_{K/k}\arrow[r] & \bigoplus\limits_{i=1}^{m}\Ind_{\Gamma_{K_i}}^{\Gamma_k}(\wh{T}_{E_i/K_i})\arrow[r]  & \wh{S}_{K,K'} \arrow[r] & 0,
\end{tikzcd}
\end{center}
where $\Gamma_k$ (resp.~$\Gamma_{K_i}$) denotes the absolute Galois group of $k$ (resp.~$K_i$). Here we write $K'=\prod_{i=1}^m K_i$ into a product of finite separable field extensions $K_i$ of $k$ and set $E_i=K\otimes_k K_i$.
This yields the following two compatible long exact sequences which form the following commutative diagram (see \eqref{eq:lg-exact}):
\begin{equation*}
 \begin{tikzcd}
H^1(k,\wh{T}_{K/k}) \arrow[r, "\iota^1"] \arrow[d, "r_K"] & {\bigoplus\limits_{i=1}^{m}H^1(K_i,\wh{T}_{E_i/K_i})} \arrow[r, "\rho^1"] \arrow[d, "r_{E/K'}"] & {H^1(k,\wh{S}_{K,K'})} \arrow[r, "\delta^1"] \arrow[d, "{r_{K,K'}}"] & H^2(k,\wh{T}_{K/k}) \arrow[d, "r_K^2"] \\
\prod\limits_v H^1(k_v,\wh{T}_{K_v/k_v}) \arrow[r, "\iota^1_a"]         & {\bigoplus\limits_{i=1}^{m}\prod\limits_v H^1(K_{i,v},\wh{T}_{E_{i,v}/K_{i,v}})} \arrow[r, "\rho^1_a"]      & {\prod\limits_v H^1(k_v,\wh{S}_{K_v,K'_v})} \arrow[r, "\delta^1_a"]                & \prod\limits_v H^2(k_v,\wh{T}_{K_v/k_v}),           
\end{tikzcd}   
\end{equation*}
where $v$ runs through all places of $k$, $X_v:=X\otimes_k k_v$ denotes the completion of a $k$-algebra $X$ at $v$, and the vertical maps are the corresponding local-global maps.
Define    
\begin{equation}
 S(K,K'):=\left\{ x\in \bigoplus_{i=1}^{m}H^1(K_i,\wh{T}_{E_i/K_i}): r_{E/K'}(x)\in \im{(\iota^1_a)} \right\}
\end{equation}
and  
\begin{equation}
        D(K,K'):=\iota^1(H^1(k,\wh{T}_{K/k})).
    \end{equation}
Clearly, $D(K,K')\subset S(K,K')$.

We show the following result; see Proposition~\ref{prop:S/D<Sha}, Corollary~\ref{cor:Sha=S/D} and Proposition~\ref{prop:Sha=S/D}. 

\begin{thm}\label{thm:I.1}
   Let $L=K\times K'$ be a finite \'etale $k$-algebra, where both $K=\prod_{j=1}^s F_j$ and $K'=\prod_{i=1}^m K_i$ are \'etale $k$-algebras with finite separable field extensions $F_j$ and $K_i$ over $k$. 
\begin{enumerate}
\item[{\rm (1)}] There is a short exact sequence of abelian groups 
\begin{equation}\label{eq:I.7}
\begin{tikzcd}
0 \arrow[r, "\iota^1"] & S(K,K')/D(K,K') \arrow[r, "\rho^1"] & {\Sha^1(k,\wh{S}_{K,K'})} \arrow[r, "\delta^1"] & \Sha^2(k,\wh{T}_{K/k}).
\end{tikzcd}    
\end{equation}

\item[{\rm (2)}] If one of the following holds:
\begin{itemize}
    \item [(i)] $K$ is a field which is a Galois extension over $k$ with Galois group $G$, and there is a place $v$ of $k$ such that the decomposition group $G_v$ of $G$ is equal to $G$; 
  \item[(ii)] The Galois closure $\wt F$ of the compositum of all $F_j$ over $k$ satisfies $\wt F\cap K_i=k$ for some~$i$,
\end{itemize}
then we have an isomorphism 
\begin{equation} \label{eq:I.8} 
\Sha^1(k,\wh{S}_{K,K'})\simeq S(K,K')/D(K,K').    
\end{equation}
\end{enumerate}
\end{thm}

When $K/k$ is cyclic, the group $S(K,K')$ coincides with the group $G(K,K')$ constructed in \cite{BLP19} in a different way; see the description of $G(K,K')$ in Section~\ref{subsec:4.3} and Lemma~\ref{lm:S=G}. Thus, Theorem~\ref{thm:I.1}(2)(i) generalizes  \cite[Proposition 5.19]{BLP19}. 

Based on Theorem~\ref{thm:I.1} we show two previous results on $\Sha(L/k)$ as follows, due to Demarche and Wei and Pollio, respectively:

\begin{thm}[{Demarche-Wei~\cite[Theorem 1]{demarche-wei}}]\label{thm:DW}
Let the notation be as in Theorem~\ref{thm:I.1}.
If $\widetilde{F}\cap (K_{1}\cdots K_{m})=k$, where $\widetilde{F}$ is the Galois closure of $F_1\cdots F_{s}$ over $k$, then $\Sha(L/k)=0$.  
\end{thm}

\begin{thm}[Pollio~{\cite[Theorem 1]{pollio}}]\label{thm:Pol}
Let $L=K_0\times K_1$, where $K_0$ and $K_1$ are finite abelian field extensions of $k$. The natural morphism $\Sha(L/k)\to \Sha(K_0\cap K_1/k)$ is an isomorphism. 
\end{thm}

We remark that Theorems \ref{thm:DW} and~\ref{thm:Pol}
were proved by different methods and we provide a more uniform proof for both theorems here. More precisely, we first give a more concrete description of $S(K,K')$ and $D(K,K')$ in terms of the character groups of related Galois groups (Section~\ref{sec:4.2}). We then verify Theorem~\ref{thm:DW} by showing that the map $\delta^1$ is zero and $S(K,K')/D(K,K')=0$. For proving Theorem~\ref{thm:Pol}, we show again the vanishing of $S(K,K')/D(K,K')$ and then prove that the image of $\delta^1$ is isomorphic to the Pontryagin dual of $\Sha(K_0\cap K_1/k)$.   

We also show the following reduction results; see Theorem~\ref{thm:G(p)xH} and Proposition~\ref{prop:nitlponent}. 

\begin{thm}\label{thm:I.2}
\begin{enumerate}
\item[\rm (1)] Assume that $K$ is a Galois extension over $k$ with Galois group $G$ which is of the form $G=H\rtimes G_p$ for a Sylow $p$-subgroup $G_p$, where $p$ is a prime number,  and a normal subgroup $H\subset G$ of prime-to-$p$ order. Then the $p$-primary component $\Sha(L/k)(p)$ of $\Sha(L/k)$ is isomorphic to  $\Sha((K(p)\times K')/k)$,
where $K(p)\subset K$ is the subfield fixed by $H$.

\item[\rm (2)] If $L=\prod_{i=0}^m K_i$ and every component $K_i$ is a nilpotent extension of $k$, then 
\[ \Sha(L/k)=\bigoplus_{p|d} \Sha(L(p)/k), \]
where $d=\gcd([K_0:k],\dots, [K_m:k])$ and $L(p)$ is the maximal $k$-subalgebra of $L$ of $p$-power degree. 
\end{enumerate}
\end{thm}

Theorem~\ref{thm:I.2} (2) is a direct consequence of Theorem~\ref{thm:I.2} (1) as every finite nilpotent group is a product of its Sylow $p$-subgroups. By Theorem~\ref{thm:I.2} (2), if every component $K_i$ of $L$ is a nilpotent extension, then without loss of generality we can assume that every component $K_i$ of $L$ is of $p$-power degree. When every factor of $L$ is cyclic, Theorem~\ref{thm:I.2}~(2) is proved in \cite[Proposition 8.6]{BLP19}, which plays a role in the proof of a main result \cite[Theorem 8.1]{BLP19}. We remark that the proof of \cite[Proposition 8.6]{BLP19} does not yield a proof of 
Theorem~\ref{thm:I.2} (2) directly. 
The main difference is that the Hasse principle, which is an essential step used in the cyclic case,  does not hold for nilponent extensions in general.

In \cite{lee:jpaa2022} Lee gave an explicit description of $\Sha^1(k, T_{L/k})$ under the condition that $\bigcap_{i=0}^m K_i=k$. 
Using the following result (also see Corollary~\ref{cor:L/FvsL/k}), one can remove this condition, by replacing $k$ by $F:=\bigcap_{i=0}^m K_i$ and applying Lee's result for $\Sha(L/F)$. Lee's result is rather technical; we refer to~\cite[Theorem 6.5]{lee:jpaa2022} or an exposition \cite[Section 2]{HHLY}.

\begin{thm}\label{thm:I.5}
Assume that a component $K_0$ of $L$ is a cyclic field extension. Set $F:=\bigcap_{i=0}^m K_i$. Then there is an isomorphism $\Sha(L/k)\simeq \Sha(L/F)$.
\end{thm}

This paper is organized as follows. Section 2 contains preliminaries and some results on multinorm-one tori, including their Tate-Shafarevich group, their first cohomological group, and Tamagawa numbers. In Section 3, we discuss a few functorial maps between cohomological groups of multinorm-one tori. In Section 4 we construct two finite groups $S(K,K')$ and $D(K,K')$ and show that they agree with the groups $G(K,K')$ and $D(K,K')$ constructed in \cite{BLP19} in the cyclic case. The proof of the structure theorem (Theorem~\ref{thm:I.1}) is given here. We then provide a more concrete description of the groups $S(K,K')$ and $D(K,K')$ in terms of group theory and show Theorem~\ref{thm:I.5}. As applications of Theorem~\ref{thm:I.1}, we give other proofs of Theorems~\ref{thm:DW} and~\ref{thm:Pol} in Section 5.
The proof of Theorem~\ref{thm:I.2} is given in Section 6. 

\section{Preliminaries on Multinorm-one Tori}\label{sec:P}

\subsection{Setting}\label{sec:P.1} Throughout this paper, $k$ denotes a global field. Let $k_s$ be a fixed separable closure of $k$, and let $\Gamma_k=\Gal(k_s/k)$ be the absolute Galois group of $k$. All finite separable field extensions of $k$ in this paper are assumed to be contained in $k_s$. 
Denote by $\Omega_k$ the set of all places of $k$. For each $v\in \Omega_k$, denote by $k_v$ the completion of $k$ at $v$. If $K/k$ is Galois with Galois group $G$, we write $G_v$ for the decomposition group of $v$, which is unique up to conjugate. For any field $K$, denote by $\Gamma_K=\Gal(K_s/K)$ the absolute Galois group of $K$ for a chosen separable closure $K_s$ of $K$; if $K$ is a field extension of $k$, $K_s$ and $k_s$ are chosen so that $k_s$ is contained in $K_s$. 
We have the following picture of field extensions:
\begin{center}
\begin{tikzcd}
               & k_s \arrow[r]  & k_sK \arrow[r]  & K_s  \\
k \arrow[r]                    & k_s\cap K \arrow[r] \arrow[u]                   & K   \arrow[u]                                        
\end{tikzcd}
\end{center}
Restriction from $Kk_s$ to $k_s$ gives an isomorphism $\Gal(Kk_s/K)\simeq \Gal(k_s/k_s\cap K)\subset \Gamma_k$. Thus, the restriction map $\sigma\mapsto \sigma|_{Kk_s}$ gives a group homomorphism (also called the restriction map)
\begin{equation}\label{eq:rKk}
    r_{K/k}: \Gamma_K \to \Gamma_k.
\end{equation}

Let $\mathbb{G}_{m,k}=\Spec k[X,X^{-1}]$ denote the split multiplicative group of $k$. For any $k$-torus $T$, denote by $\widehat{T}:=\Hom_{k_s}(T,\mathbb{G}_{m,k})$ the character group of $T$ over $k_s$; $\wh T$ is a finite and free $\Z$-module together with a continuous action of $\Gamma_k$. 

For any group $G$ and $G$-module $A$, denote by $H^i(G,A)$, $i\ge 0$, the $i$th Galois cohomology of $G$ with coefficients in $A$.
Let 
\[ H^i(k, T):=H^i(\Gamma_k, T)=\varinjlim_{K} H^i(\Gal(K/k), T) \]
be the $i$th cohomology group of $T$, where $K$ runs over all Galois extensions of $k$. 
If $K/k$ is a Galois splitting field of $T$ with Galois group $G$, then $H^1(G, T)$ is independent of the choice of the splitting field $K$ and is isomorphic to $H^1(k, T)$.

Let 
\begin{equation} \label{eq:ShaT}
  \Sha^i(k, T):=\Ker \left ( H^i(k,T)\to \prod_{v\in \Omega_k} H^i(k_v, T) \right )   
\end{equation}
be the $i$th Tate--Shafarevich group of $T$. 
We also define 
\[ \Sha^i_{\omega}(k, T):=\{ [C]\in H^i(k,T): [C]_v=0 \text{ for almost all $v\in \Omega_k$}  \}, \]
where $[C]_v$ is the image of the class $[C]$ in $H^i(k_v,T)$. Similarly, the group $\Sha^1(k, T)$ is isomorphic to $\Sha^1(G,T)$, where $K/k$ is a splitting Galois extension with group $G=\Gal(K/k)$, and $\Sha^1(G,T):=\Ker (H^1(G,T)\to \prod_{v} H^1(G_v,T))$.

By Poitou-Tate duality~\cite[Theorem 6.10]{platonov-rapinchuk}, there is a natural isomorphism 
\[\Sha^1(G,T)^{\vee} \cong \Sha^2(G,\widehat{T}).\]
Here for any finite abelian group $A$, $A^\vee$ denotes the Pontryagin dual $\Hom(A, \Q/\Z)$. 
Hence the group $\Sha^2(G,\widehat{T})$ is also independent of the choice of the splitting field $K/k$. 
One can also show that $\Sha_\omega^2(G,\wh T)$ does not depend on the choice of the splitting field.


\subsection{Multinorm-one Tori}\label{sec:TLk}

Let $L=\prod_{i=0}^{m}K_i$ be a finite \'etale algebra over $k$, that is, 
it is a product of finite separable field extensions $K_i/k$. 
Let $T^{L}$ denote the algebraic torus over $k$ such that $T^L(R)=(L\otimes_k R)^{\times}$ for any commutative $k$-algebra $R$. (The torus $T^L$ is denoted by $R_{L/k}(\G_{{\rm m},L})$ in \cite{BLP19}.)
Let $N_{L/k}\colon T^L\to \Gmk$ be the norm map. The induced map $N_{L/k}\colon T^L(k)=L^\times \to k^\times$ is given by
$$N_{L/k}((x_i))=\prod_{i=0}^{m}N_{K_i/k}(x_i)$$
where $x_i\in K_i^\times$.

\begin{defn}
The \textit{multinorm-one torus} $T_{L/k}$ associated to $L/k$ is defined to be the $k$-torus $T_{L/k}:=\Ker(N_{L/k})$, the kernel of the norm $N_{L/k}$.     
\end{defn}

According to the definition, we have an exact sequence of $k$-tori
\begin{equation}\label{eq:exact_TLk}
    \begin{tikzcd}
    1 \arrow[r] & T_{L/k} \arrow[r, "\iota"] & T^L  \arrow[r, "N"] & \mathbb{G}_{m, k} \arrow[r] & 1.
    \end{tikzcd}
\end{equation}

By Hilbert's Theorem 90, we obtain an isomorphism from~\ref{eq:exact_TLk}
\begin{equation}\label{eq:2.3}
   k^\times/N_{L/k}(L^\times) \xrightarrow[]{\,\sim\,} H^1(k, T_{L/k}), 
\end{equation}
and this isomorphism yields the following commutative diagram:
\[ \begin{CD}
k^\times/N_{L/k}(L^\times) @>\sim >>   H^1(k, T_{L/k}) \\
@VVV @VVV \\
\prod_{v} k_v^\times/N_{L_v/k_v}(L_v^\times) @>\sim>> \prod_v H^1(k_v, T_{L_v/k_v}).     
\end{CD}
\]
Taking the kernel of each vertical (local-global) map, 
we have the isomorphisms
\begin{equation}\label{eq:2.4}
   \Sha(L/k) \simeq \Sha^1(k, T_{L/k}) \simeq \Sha^2(k, \widehat{T}_{L/k})^{\vee}, 
\end{equation}
where the second isomorphism is given by Poitou--Tate duality.

\begin{lem}\label{lm:reduce_comp}
Let $L=\prod_{i=0}^m K_i$ be a finite \'etale $k$-algebra with $m\ge 1$ and $L':=\prod_{i=0}^{m-1} K_i$. Assume $K_{m-1}\subset K_m$. 
\begin{enumerate}
\item[{\rm (1)}] There is an isomorphism $T_{L/k}\simeq T_{L'/k} \times T^{K_m}$.
\item[{\rm (2)}] There are isomorphisms $H^1(k, \wh T_{L/k})\simeq H^1(k, \wh T_{L'/k})$ and $\Sha ^2(k, \wh T_{L/k})\simeq \Sha^2(k, \wh T_{L'/k})$.    
\end{enumerate}
\end{lem}
\begin{proof}
(1) Define an isomorphism $\xi \colon T^L \isoto T^{L'} \times T^{K_{m}}$ by sending 
\[ x=(x_i) \mapsto (x', x_m), \quad x'=(x_0, \dots, x_{m-2}, x_{m-1} N_{K_{m}/K_{m-1}}(x_m)). \] Then we have $N_{L/k}(x)=1$ if and only if $N_{L'/k}(x')=1$. So the morphism $\xi$ induces an isomorphism $T_{L/k}\isoto T_{L'/k}\times T^{K_m}$.

(2) By Shapiro's lemma, $H^1(k, \wh {T}^{K_m})\simeq H^1(K_m, \Z)=0$. One also has $H^1(k, T^{K_m})\simeq H^1(K_m, \Gm)=0$ by Hilbert's Theorem 90. Therefore, 
\[ H^{1}(k,\widehat{T}_{L/k})\simeq H^{1}(k,\widehat{T}_{L'/k})\times H^{1}(k,\widehat{T}^{K_m})=H^{1}(k,\widehat{T}_{L'/k}).\] 
Also, the isomorphism
\[ H^{1}(k,T_{L/k})\simeq H^{1}(k,T_{L'/k})\times H^{1}(k,T^{K_m})=H^{1}(k,T_{L'/k}) \] implies 
$\Sha ^1(k, T_{L/k})\simeq \Sha^1(k, T_{L'/k})$ and we obtain $\Sha ^2(k, \wh T_{L/k})\simeq \Sha^2(k, \wh T_{L'/k})$ by \eqref{eq:2.4}. \qed
\end{proof}

By Lemma~\ref{lm:reduce_comp}, we have the following result (cf.~\cite[Lemma 1.1]{BLP19}).

\begin{cor} \label{cor:H1Kn}
If $L=K^n$ is the product of $n$ copies of a finite separable field extension $K/k$, then there exist isomorphisms $H^1(k, \wh T_{L/k})\simeq H^1(k, \wh T_{K/k})$ and $\Sha ^2(k, \wh T_{L/k})\simeq \Sha^2(k, \wh T_{K/k})$.
\end{cor}

\subsection{The first cohomology group}

\begin{prop}\label{prop:2.4}
Let $L=\prod_{i=0}^{m}K_i$ and $T=T_{L/k}$, the associated multinorm-one torus. Then there is a natural isomorphism 
\[  H^1 \big(k, \widehat{T}\big) \simeq \Gal(F_{ab}/k)^\vee. \] 
where $F_{ab}$ is the maximal abelian extension of $k$ contained in all $K_i$. In particular, when $K/k$ is a Galois extension, one has a natural isomorphism 
\begin{equation}\label{eq:2.6}
    H^1(k, \wh T_{K/k})\isoto \Hom(\Gal(K/k),\Q/\Z). 
\end{equation}
\end{prop}

\begin{proof}
Let $\tilde{K}$ be the Galois closure of the compositum $K_0\cdots K_m$ over $k$ in $k_s$, $G=\Gal(\tilde{K}/k)$, 
and $H_i=\Gal(\tilde{K}/K_i)$. Taking the dual of the exact sequence \eqref{eq:exact_TLk},
we obtain the exact sequence 
$$ 0\longrightarrow \mathbb{Z}\longrightarrow \widehat{T}^L
\longrightarrow\widehat{T}\longrightarrow 0$$
and hence the long exact sequence 
\begin{equation}\label{eq:H1T_hat}
H^1(G,\widehat{T}^L)\longrightarrow H^1(G,\widehat{T})\longrightarrow H^2(G,\mathbb{Z})\longrightarrow H^2(G,\widehat{T}^L).
\end{equation}

We have $\widehat{T}^L=\bigoplus_{i=0}^{m}\Ind_{H_i}^{G}\mathbb{Z}$. It follows from Shapiro's lemma that $H^1(G,\widehat{T}^L)\simeq \bigoplus_{i=0}^{m} H^1(H_i, \Z)=0$. Using the canonical isomorphisms $H^2(G,\Z)\simeq H^1(G,\Q/\Z)\simeq \Hom(G,\Q/\Z)$ and $H^2(G,\wh{T}^L)\simeq \oplus_i \Hom(H_i, \Q/\Z)$,
we get from \eqref{eq:H1T_hat} that 
\begin{equation}\label{eq:2.8}
  H^1(G,\widehat{T})\simeq \Ker(\Hom(G,\mathbb{Q/Z})\longrightarrow  \bigoplus_{i=0}^m \Hom(H_i,\mathbb{Q/Z})).  
\end{equation}
Let $N$ be the normal subgroup of $G$ generated by the subgroups $H_0,...,H_m$ and $[G,G]$. We have $G/N\simeq \Gal(F_{ab}/k)$.
Then 
$$H^1(G,\widehat{T})\simeq \Hom(G/N,\mathbb{Q/Z})\simeq\Hom(\Gal(F_{ab}/k),\mathbb{Q/Z}).$$ 
This proves the proposition. \qed
\end{proof}

\section{Functorial maps}

\begin{lem}
For any finite separable field extensions $k\subset M\subset K$, we have the following commutative diagram of $k$-tori:

\begin{center}

\begin{tikzcd}
0 \arrow[r] & R_{M/k}(T_{K/M}) \arrow[r, hook, "{\rm id}"] & T_{K/k} \arrow[r, "N_{K/M}"]                      & T_{M/k} \arrow[r] & 0 \\
            &                            & T_{M/k} \arrow[u, "\iota"] \arrow[ru, "{[K:M]}"'] &                   &   \\
            &                            & 0 \arrow[u]                                       &                   &  
\end{tikzcd}
\end{center}
This induces the following morphisms:
\begin{equation}\label{eq:func_map}
\begin{array}{c}
   \wh \iota \colon H^i(k,\widehat{T}_{K/k})\rightarrow H^i(k,\widehat{T}_{M/k}), \\ 
  \wh N_{K/M} \colon H^i(k,\widehat{T}_{M/k})\rightarrow H^i(k,\widehat{T}_{K/k}), \\  
  \wh \id  \colon H^i(k,\widehat{T}_{K/k})\rightarrow H^i(M,\widehat{T}_{K/M}).    
\end{array}
\end{equation}
\end{lem}

\begin{proof}
    Let $N_{K/M}\colon T^K \to T^M$ be the norm map. For $x\in K^\times$, we put $y=N_{K/M}(x)$ and have $N_{K/k}(x)=N_{M/k} N_{K/M}(x)=N_{M/k}(y)$. Thus, $N_{K/k}(x)=1$ if and only if $N_{M/k}(y)=1$. Therefore, we have the cartesian diagram: 
\[ \begin{CD}
T_{K/k} @>>> T^K \\
@VV{N_{K/M}}V  @VV{N_{K/M}}V  \\
T_{M/k} @>>> T^M. 
\end{CD} \]    
    It is clear that the kernel of $N_{K/M}$ is equal to $R_{M/k}(T_{K/M})$. 

    Let $\iota \colon T^M \to T^{K}$ be the monomorphism induced by the inclusion map $M^\times \subset K^\times$. 
    For $x\in M^\times$, we have
\begin{equation}\label{eq:iotahat}
    N_{K/k}(x)=N_{M/k} N_{K/M}(x)=N_{M/k}(x)^{[K:M]}.
\end{equation}    
     Thus, the condition $N_{M/k}(x)=1$ implies that $N_{K/k}(x)=1$, so that $\iota$ induces a monomorphism $\iota \colon T_{M/k}\to T_{K/k}$. 
     
     These maps induce the respective maps
\[ \wh \iota: \widehat{T}_{K/k}\longrightarrow \widehat{T}_{M/k}, \quad  
  \wh N_{K/M}: \widehat{T}_{M/k}\longrightarrow \widehat{T}_{K/k}, \quad  
  \wh \id  : \widehat{T}_{K/k} \longrightarrow \widehat{R_{M/k} ({T}_{K/M})}, \]            
  and the respective maps in \eqref{eq:func_map}.   \qed
\end{proof}

\begin{lem}\label{lm:b} \rm(Base change map)
Let $L$ be a finite \'etale $k$-algebra and $K/k$ any field extension. There exist a canonical isomorphism 
\begin{equation}\label{eq:base_change}
    T_{L_K/K}\simeq T_{L/k}\otimes_k K.
\end{equation}
and a base change map
\[
b_{K/k}\colon H^i(k,\widehat{T}_{L/k})\rightarrow H^i(K,\widehat{T}_{L_{K}/K})\simeq H^i(K,\widehat{T}_{L/k}),
\] 
where $L_{K}:=L\otimes_k K$.
\end{lem}

\begin{proof}
Recall we have the restriction map $r_{K/k}\colon \Gamma_K:=\Gal(K_s/K)\to \Gamma_k$ through sending $\sigma\mapsto \sigma|_{k_s K}$ and an isomorphism $\Gal(k_sK/K)\simeq \Gal(k_s/k_s \cap K)$. 
For any commutative $K$-algebra $R$, $T_{L_K/K}(R)=\{x\in (L_K\otimes_K R)^\times: N_{L/k}(x)=1\}$. Since $(L_K\otimes_K R)^\times=(L\otimes_k R)^\times$, we see $T_{L_K/K}(R)=T_{L/k}(R)$ where in $T_{L/k}(R)$ we regard $R$ as a $k$-algebra via the inclusion $k\embed K$. Thus, as a group functor, $T_{L_K/K}$ is the restriction of $T_{L/k}$ from the category of commutative $k$-algebras to that of $K$-algebras. This gives a canonical isomorphism $T_{L_K/K}\simeq T_{L/k}\otimes_k K$. Consequently, the character group $\wh {T}_{L_K/K}$ is equal to $\wh {T}_{L/k}$, 
which is regarded as a $\Gamma_K$-module through the homomorphism $r_{K/k}\colon \Gamma_K \to \Gamma_k$, and we have a canonical isomorphism 
\[ H^i(K, \wh {T}_{L_K/K})\simeq H^i(K, \wh {T}_{L/k}). \]
By inflation, we obtain a map $b_{K/k}\colon H^i(k, \wh {T}_{L
/k}) \to H^i(K, \wh {T}_{L/k})\simeq H^i(K, \wh {T}_{L_K/K})$. \qed
\end{proof}

\begin{lem}\label{lm:base_change}
Let $K=\prod_{i=1}^s F_j$ be a finite \'etale $k$-algebra over $k$, where $F_1,\ldots,F_s$ are finite separable field extensions of $k$. Denote by $\wt F$ the Galois closure of the compositum of $F_1,\ldots,F_s$ over $k$.  Consider a field extension $K'/k$, and fix a separable closure $K'_s$ of $K'$ containing $k_s$.
We write $M$ for the compositum of $\wt F$ and $K'$, which is a finite Galois extension over $K'$. 
Then one has a commutative diagram 

\begin{center}
\begin{tikzcd}
H^1(k, \wh {T}_{K/k}) \arrow[r,"b_{K'/k}"] \arrow[hook]{d} & H^1(K', \wh {T}_{K/k}) \arrow[hook]{d} \\
\Hom(\Gal(\wt F /k),\Q/\Z)  \arrow[r,"r"]  & \Hom(\Gal(M/K'), \Q/\Z),       
\end{tikzcd}
\end{center}
where $r$ is the restriction map obtained by the injection map $\Gal(M/K') \to \Gal(\wt F/k)$. 
Furthermore, if $K$ is a field which is Galois over $k$, then the vertical homomorphisms are isomorphisms.
\end{lem}

\begin{proof}   
Set $G:=\Gal(\wt F/k)$. By Proposition~\ref{prop:2.4},
there is a canonical injection $H^1(k, \wh {T}_{K/k})\embed \Hom(G,\Q/\Z)$. Note that $M/K'$ is a Galois splitting field for the $K'$-torus $T_{K\otimes K'/K'}$. By Lemma~\ref{lm:b} and Proposition~\ref{prop:2.4}, we obtain 
\[ H^1(K', \wh {T}_{K/k})\simeq H^1(K', \wh {T}_{K\otimes K'/K'})\embed \Hom(\Gal(M/K'),\Q/\Z), \]
which is our right vertical injection.  

As shown in Lemma~\ref{lm:b}, we know that the $\Gal(M/K')$-module $\wh T_{K\otimes K'/K'}$ is the $\Gal(M/K')$-module $\wh T_{K/k}$ obtained through the inclusion map $\Gal(M/K')\subset G$, and that the map $b_{K'/k}\colon$ $ H^1(G,\wh T_{K/k}) \to H^1(\Gal(M/K'), \wh T_{K/k})$ is simply the restriction map. 
So $b_{K'/k}$ is compatible with the map $r$ and that makes the diagram commutative.  Moreover, these vertical injections are isomorphisms when $K/k$ is a Galois field extension by Proposition~\ref{prop:2.4}.\qed


\end{proof}

\section{The construction of the torus $S_{K,K'}$ and its Tate--Shafarevich group}\label{sec:C}

\subsection{The groups $S(K,K')$ and $D(K,K')$}\label{ss:sddf}

Let $L=K\times K'$, where $K=\prod_{j=1}^s F_j$ and $K'=\prod_{i=1}^{m}K_i$ are finite \'etale $k$-algebras with each $F_j$ or $K_i$ a finite separable field extension of $k$. Set 
\[ \text{ $E:=K\otimes_{k} K'$ \quad and \quad $E_i:=K\otimes_{k} K_i$.} \]
Recall that $S_{K,K'}$ is the $k$-torus defined as the kernel of $f$ in~\eqref{eq:f}. 
From \cite[Section 2]{BLP19}, one has 
an exact sequence of $k$-tori:
\begin{equation}\label{eq:Def of S_{K,K'}}
\begin{tikzcd}
 0 \arrow[r] & S_{K,K'} \arrow[r] & R_{E/k}(\mathbb{G}_{m,E}) \arrow[r, "f"] & T_{L/k} \arrow[r] & 0. 
\end{tikzcd}
\end{equation}
The $k$-torus $S_{K,K'}$ also fits in the exact sequence 
\begin{equation}\label{eq:S_{K,K'}}
\begin{tikzcd}
1\arrow[r] & S_{K,K'} \arrow[r] & R_{K'/k}(T_{E/K'}) \arrow[r,"{N_{E/K}}"] &T_{K/k} \arrow[r] & 1. 
\end{tikzcd}
\end{equation}
The first exact sequence \eqref{eq:Def of S_{K,K'}} gives rise to the isomorphism  \cite[Lemma 3.1]{BLP19}
\begin{equation}\label{prop: ShaT &ShaS}
  \Sha^1(k, \widehat{S}_{K,K'}) \isoto \Sha^2(k, \wh{T}_{L/k}).
\end{equation}

Taking the dual of \eqref{eq:S_{K,K'}}, we obtain the exact sequence:
\begin{center}

\begin{tikzcd}
0 \arrow[r] &\wh{T}_{K/k}\arrow[r] & \bigoplus\limits_{i=1}^{m}\Ind_{\Gamma_{K_i}}^{\Gamma_k}(\wh{T}_{E_i/K_i})\arrow[r]  & \wh{S}_{K,K'} \arrow[r] & 0,
\end{tikzcd}
\end{center}
and hence the long exact sequence:
\begin{center}
\begin{tikzcd}
H^1(k,\wh{T}_{K/k}) \arrow[r, "\iota^1"] & {\bigoplus_{i=1}^{m}H^1(K_i,\wh{T}_{E_i/K_i})} \arrow[r, "\rho^1"] & {H^1(k,\wh{S}_{K,K'})} \arrow[r, "\delta^1"] & H^2(k,\wh{T}_{K/k}). 
\end{tikzcd}
\end{center}
Taking the local-global maps into account, we have the following commutative diagram:
\begin{equation}\label{eq:lg-exact}
 \begin{tikzcd}
H^1(k,\wh{T}_{K/k}) \arrow[r, "\iota^1"] \arrow[d, "r_K"] & {\bigoplus\limits_{i=1}^{m}H^1(K_i,\wh{T}_{E_i/K_i})} \arrow[r, "\rho^1"] \arrow[d, "r_{E/K'}"] & {H^1(k,\wh{S}_{K,K'})} \arrow[r, "\delta^1"] \arrow[d, "{r_{K,K'}}"] & H^2(k,\wh{T}_{K/k}) \arrow[d, "r_K^2"] \\
\prod\limits_v H^1(k_v,\wh{T}_{K_v/k_v}) \arrow[r, "\iota^1_a"]         & {\bigoplus\limits_{i=1}^{m}\prod\limits_v H^1(K_{i,v},\wh{T}_{E_{i,v}/K_{i,v}})} \arrow[r, "\rho^1_a"]      & {\prod\limits_v H^1(k_v,\wh{S}_{K_v,K'_v})} \arrow[r, "\delta^1_a"]                & \prod\limits_v H^2(k_v,\wh{T}_{K_v/k_v}).            
\end{tikzcd}   
\end{equation}

\begin{defn}\label{def:SKK'}
We define    
\begin{equation}\label{eq:SKK'}
 S(K,K'):=\left\{ x\in \bigoplus_{i=1}^{m}H^1(K_i,\wh{T}_{E_i/K_i}): r_{E/K'}(x)\in \im{(\iota^1_a)} \right\}
\end{equation}
and  
\begin{equation}
        D(K,K'):=\iota^1(H^1(k,\wh{T}_{K/k})).
    \end{equation}
\end{defn}

The group $S(K,K')$ plays the role of a Selmer group in $\bigoplus_{i=1}^{m}H^1(K_i,\wh{T}_{E_i/K_i})$.
It is easy to see that the group $S(K,K')$ is equal to the pre-image of $\Sha^1(k,\wh{S}_{K,K'})$ under the map $\rho^1$. So we arrive at an exact sequence of abelian groups:
\begin{equation}\label{eq:Sha1S_KK'1}
\begin{tikzcd}
H^1(k,\wh{T}_{K/k}) \arrow[r, "\iota^1"] & S(K,K') \arrow[r, "\rho^1"] & {\Sha^1(k,\wh{S}_{K,K'})} \arrow[r, "\delta^1"] & \Sha^2(k,\wh{T}_{K/k}). 
\end{tikzcd}    
\end{equation}
As a result from \eqref{eq:Sha1S_KK'1}, we show the following proposition.

\begin{prop}[=Theorem~\ref{thm:I.1}(1)]\label{prop:S/D<Sha}
There is a short exact sequence of abelian groups 
\begin{equation}\label{eq:Sha1S_KK'2}
\begin{tikzcd}
0 \arrow[r, "\iota^1"] & S(K,K')/D(K,K') \arrow[r, "\rho^1"] & {\Sha^1(k,\wh{S}_{K,K'})} \arrow[r, "\delta^1"] & 
\Sha^2(k,\wh{T}_{K/k}). 
\end{tikzcd}    
\end{equation}
In particular, if $\Sha^{2}(k,\wh{T}_{K/k})=0$ then 
\begin{equation}\label{eq:Sha1S_KK'3}
\Sha^1(k,\wh{S}_{K,K'})\simeq S(K,K')/D(K,K').    
\end{equation}
\end{prop}

\begin{rmk}\label{rem:S/D=Sha}
\begin{enumerate}
\item The group $S(K,K')$ is a generalization of $G(K,K')$ in \cite{BLP19} constructed under the assumption that $K$ is cyclic over $k$. Actually, the authors of 
loc.~cit.~first constructed the group $G(K,K')$ in a different way and proved that $G(K,K')/D(K,K')\simeq \Sha^1(k,\wh{S}_{K,K'})$. See also Lemma \ref{lm:S=G} below. 
\item Set  
   \[ \delta(K,K'):=\delta^1(\Sha^1(k,\wh{S}_{K,K'})).  \]
Since $\Sha^1(k,\wh{S}_{K,K'})\simeq \Sha(L/k)^\vee$, we get the short exact sequence by \eqref{eq:Sha1S_KK'2})
\begin{equation}
    \label{eq:s_exact_ShaL}
\begin{tikzcd}
0 \arrow[r] & \delta(K,K')^\vee \arrow[r, "{\delta^1}^\vee"]  & \Sha(L/k) \arrow[r, "{\rho^1}^\vee"] & \left ( S(K,K')/D(K,K') \right )^\vee
 \arrow[r] & 0. 
\end{tikzcd}        
\end{equation}
It is interesting to know whether 
these subquotients $\delta(K,K')^\vee$ and $\left ( S(K,K')/D(K,K') \right )^\vee$ have any interpretation in the obstruction group $\Sha(L/k)$. 
\end{enumerate}
\end{rmk}

As shown in \cite{BLP19} for some special cases, the group $S(K,K')$ plays a crucial role in the computation of ${\Sha^1(k,\wh{S}_{K,K'})}$.  
Thus, investigating sufficient conditions for which Equation~\eqref{eq:Sha1S_KK'3} holds is helpful. 

Recall that Tate's theorem~\cite[Theorem 6.11, p.~308]{platonov-rapinchuk} says that if $K/k$ is a Galois field extension with Galois group $G$, then we have the isomorphisms  
\[ \Sha^2(k,\wh T_{K/k})\simeq \Sha^3(G, \Z)\simeq \Sha^2(G,\Q/\Z). \] 

\begin{cor}\label{cor:Sha=S/D}
   Let $L=K\times K'$ be a finite \'etale $k$-algebra with $K'=\prod_{i=1}^m K_i$ a product of finite separable field extensions $K_i/k$. If $K$ is a Galois field extension over $k$ with Galois group $G$ and $G$ is equal to the decomposition group $G_{v_0}$ for some place $v_0$ of $k$ (for example, if $K/k$ is a cyclic extension), then we have an isomorphism 
\begin{equation*} 
\Sha^1(k,\wh{S}_{K,K'})\simeq S(K,K')/D(K,K').    
\end{equation*}
\end{cor}
\begin{proof}
  By Proposition~\ref{prop:S/D<Sha}, it suffices to show $\Sha^2(k,\wh T_{K/k})=0$. 
Since $G=G_{v_0}$, we get
\[\Sha^2(G,\Q/\Z)=\Ker\left (H^2(G,\Q/\Z) \to \prod_v H^2(G_v,\Q/\Z)\right )=0.\]
Therefore, by Tate's theorem we obtain $\Sha^2(k,\wh T_{K/k})\simeq\Sha^2(G,\Q/\Z)=0$. \qed   
\end{proof}

\begin{lem}\label{lem:rsif}
Let $G$ be a finite group, and $N$ be a normal subgroup. Consider a $G/N$-module $M$ and a subgroup $H$ of $G$. Then, for any positive integer $q$, there is a commutative diagram
\begin{equation*}
\xymatrix@C=55pt{
H^{q}(G/N,M)\ar[r]^{\Res_{G/N,HN/N}\hspace{2mm}}\ar[d]^{\Inf_{G,G/N}}&
H^{q}(HN/N,M)\ar[r]^{\sim \hspace{5mm}}\ar[d]^{\Inf_{HN,HN/N}}& 
H^{q}(H/(N\cap H),M)\ar[d]^{\Inf_{H,H/(N\cap H)}}\\
H^{q}(G,M)\ar[r]^{\Res_{G,HN}}& 
H^{q}(HN,M)\ar[r]^{\Res_{HN,H}}&H^{q}(H,M), }
\end{equation*}
where the right upper horizontal isomorphism is induced by the natural isomorphism
\begin{equation*}
H/(H\cap N)\simeq HN/N. 
\end{equation*}
Moreover, if $H^i(N\cap H, M)=0$ for all $0< i <q$, then the inflation map $\Inf_{H, H/(N\cap H)}$ is injective. 
\end{lem}
\begin{proof}
   Commutativity of the diagram is clear. The last assertion is a consequence of the Hochschild--Serre spectral sequence; see \cite[Proposition 5.2]{LOYY} for example. \qed  
\end{proof}

Recall that $L=K\times K'$, where
$K=\prod_{j=1}^s F_j$ and $K'=\prod_{i=1}^m K_i$ are products of finite separable field extensions $F_j$ and $K_i/k$, respectively.
Let $\wt F$ be the Galois closure of the compositum of all $F_1$, \dots, $F_s$ over $k$, and 
write $\widetilde{K}/k$ for the Galois closure of the compositum $\widetilde{F}K_{1}\cdots K_{m}/k$. Put 
\begin{equation}\label{eq:4.10}
\widetilde{G}:=\Gal(\widetilde{K}/k) \quad \text{and} \quad G:=\Gal(\widetilde{F}/k).    
\end{equation}
Furthermore, for each $i\in \{1,\ldots,m\}$, set 
\begin{equation}\label{eq:4.11}
\widetilde{N}_{i}:=\Gal(\widetilde{K}/\widetilde{F}\cap K_{i}),\quad \widetilde{H}_{i}:=\Gal(\widetilde{K}/K_{i})\quad \!\text{and} \!\quad H_{i}:=\Gal(\widetilde{F}/\widetilde{F}\cap K_{i})\simeq \Gal(\wt F K_i/K_i).
\end{equation}

\begin{prop}\label{prop:rsqt}
Let $L=K\times K'=(\prod_{j=1}^s F_j)\times (\prod_{i=1}^m K_i)$ be an \'{e}tale $k$-algebra and the notation be as above. 
Then there is a commutative diagram
\begin{equation*}
\xymatrix@C=45pt{
\Sha^{2}(k,\widehat{T}_{K/k})\ar[r]^{(b_{\widetilde{F}\cap K_i/k})_{i}\hspace{11mm}}\ar@{=}[d]& 
\bigoplus_{i=1}^{m}\Sha^{2}(\widetilde{F}\cap K_{i},\widehat{T}_{K/k})\ar[d]^{\bigoplus_{i}b_{K_{i}/\widetilde{F}\cap K_{i}}} \\
\Sha^{2}(k,\widehat{T}_{K/k})\ar[r]^{\iota^{2}\hspace{7mm}}& 
\bigoplus_{i=1}^{m}\Sha^{2}(K_{i},\widehat{T}_{K/k}). }
\end{equation*}
Moreover, the right vertical homomorphism is injective. 
\end{prop}

\begin{proof}
Using the notation in \eqref{eq:4.10} and \eqref{eq:4.11}, for each $i$, the restriction map $\widetilde{G}\rightarrow G$ induces a surjection $\widetilde{H}_{i}\twoheadrightarrow H_{i}$. Applying Lemma~\ref{lem:rsif} with groups $G,G/N, H, N$ replaced by $\wt G, G, \wt H_i, \Ker(\wt G \to G)$, respectively, we have a commutative diagram 
\begin{equation}\label{eq:cmdq}
\xymatrix@C=45pt{
H^{2}(G,\widehat{T}_{K/k})\ar[r]^{(\Res_{G,H_i})_{i}\hspace{7mm}}\ar[d]^{\Inf_{\widetilde{G},G}}& 
\bigoplus_{i=1}^{m}H^{2}(H_{i},\widehat{T}_{K/k})\ar[d]^{\bigoplus_{i}(\Res_{\widetilde{N}_i,\widetilde{H}_i}\circ \Inf_{\widetilde{N}_i,H_i})=\bigoplus_{i}(\Inf_{\wt{H_i},H_i})}\\
H^{2}(\widetilde{G},\widehat{T}_{K/k})\ar[r]^{(\Res_{\widetilde{G},\widetilde{H}_i})_{i}\hspace{7mm}}& 
\bigoplus_{i=1}^{m}H^{2}(\widetilde{H}_{i},\widehat{T}_{K/k}). 
}
\end{equation}
Recall from Section~\ref{sec:P.1} that we have the following:
\begin{equation*}
\Sha^{2}(k,\widehat{T}_{K/k})\subset H^{2}(\widetilde{G},\widehat{T}_{K/k}),\quad\!
\Sha^{2}(\widetilde{F}\cap K_{i},\widehat{T}_{K/k})\subset H^{2}(H_{i},\widehat{T}_{K/k}),\quad\! 
\Sha^{2}(K_{i},\widehat{T}_{K/k})\subset H^{2}(\widetilde{H}_{i},\widehat{T}_{K/k}). 
\end{equation*}
Moreover, the first inclusion factors as $\Sha^{2}(k,\widehat{T}_{K/k})\subset H^{2}(G,\widehat{T}_{K/k})\xrightarrow{\Inf} H^{2}(\widetilde{G},\widehat{T}_{K/k})$. Therefore the commutativity of \eqref{eq:cmdq} follows from Lemma \ref{lem:rsif}. Note that the injectivity assertion follows from the fact that $\Inf_{\widetilde{H}_i,H_i}$ is injective, as $H^1(\wt F K_i, \wh T_{K/k})=0$.  \qed
\end{proof}

\begin{prop} \label{prop:Sha=S/D}
Let $L=K\times K'$ be an \'etale $k$-algebra, where $K=\prod_{j=1}^s F_j$ and $K'=\prod_{i=1}^m K_i$ are products of finite separable field extensions $F_j$ and $K_i/k$, respectively. Let $\wt F$ be the Galois closure of the compositum $F_1 \cdots F_s$ over $k$. If $\wt F\cap K_i=k$ for some $i$, then we have an isomorphism 
\begin{equation*} 
\Sha^1(k,\wh{S}_{K,K'})\simeq S(K,K')/D(K,K').    
\end{equation*}    
\end{prop}

\begin{proof}
Since $\widetilde{F}\cap K_{i}=k$, the homomorphism $b_{K_{i}/k}$ is injective by Proposition \ref{prop:rsqt}. This immediately implies the injectivity of $\iota_{i}^{2}$. 
In particular, the map $\iota^2$ is also injective. One can replace in Proposition \ref{prop:S/D<Sha} the group $\Sha^2(k, \wh T_{K/k})$ by $\Ker(\Sha^2(k, \wh T_{K/k})\to \bigoplus_{i}H^2(K_i,\wh T_{E_i/{K_i}}))$.
Therefore, 
the map $\delta^1: \Sha^1(k,\wh S_{K,K'})\to \Sha^2(k, \wh T_{K/k})$ is zero.  The assertion then follows from Proposition~\ref{prop:S/D<Sha}. \qed 

\end{proof}

\subsection{Descriptions of $S(K,K')$ and $D(K,K')$ in terms of character groups}\label{sec:4.2}

We give a description of the finite abelian groups $S(K,K')$ and $D(K,K')$ (see Definition~\ref{def:SKK'}) by means of subgroups of a direct sum of character groups of certain Galois groups related to $L$. Let $L$, ${\wt F}$ and $G$ be as in Section \ref{ss:sddf}, that is, 
\begin{itemize}
\item $L=K\times K'$, where $K=\prod_{j=1}^{s}F_{j}$,  $K'=\prod_{i=1}^{m}K_{i}$, and $F_{j}$ and $K_{i}$ are finite separable field extensions of $k$; 
\item $\widetilde{F}/k$ is the Galois closure of the compositum $F_1\cdots F_s$ over $k$ with group $G$. 
\end{itemize}
For each $i\in \{1,\ldots,m\}$, put
\begin{equation*}
M_{i}:={\wt F}K_{i}\quad \text{and}\quad H_{i}:=\Gal(M_{i}/K_{i}). 
\end{equation*}
Then the restriction map induces an isomorphism $H_{i}\simeq \Gal({\wt F}/{\wt F}\cap K_{i})$. In particular, it gives the map 
\begin{equation}\label{eq:}
    r=(r_i)_{i}\colon \Hom(G,\Q/\Z) \longrightarrow \bigoplus\limits_{i=1}^{m}\Hom(H_i,\Q/\Z).
\end{equation}
Note that we have a commutative diagram
\begin{equation}\label{eq:nat_isom}
\begin{tikzcd}
H^1(k,\wh {T}_{K/k}) \arrow[r,"\iota^{1}"] \arrow[hook]{d} & \bigoplus_{i=1}^{m}H^1(K_i, \wh {T}_{K/k}) \arrow[hook]{d} \\
\Hom(G,\Q/\Z)  \arrow[r,"r"]  & \bigoplus_{i=1}^{m}\Hom(H_i, \Q/\Z),       
\end{tikzcd}
\end{equation}
which follows from Lemma \ref{lm:base_change}. 

\begin{defn}\label{def:diagonal}
    We say an element $x\in \bigoplus\limits_{i=1}^{m}\Hom(H_i,\Q/\Z)$ is {\it diagonal} if it lies in the image of $r$. Denote by \[ D_{\rm diag}(K,K'):=\im(r) \subset \bigoplus\limits_{i=1}^{m}\Hom(H_i,\Q/\Z) \] the subgroup consisting of all diagonal elements.  
\end{defn}

For any place $v$ of $k$ and $w_i$ of $K_i$ with $w_i|v$, write $G_v$ and $H_{i,w_i}$ for the decomposition groups of $v$ and $w_i$, respectively. Similarly for each $v$, we have the map 
\begin{equation}\label{eq:I}
    r_v:=(r_{i,w_i})_{i,w_i}\colon \Hom(G_v,\Q/\Z) \longrightarrow \bigoplus\limits_{i=1}^{m}\bigoplus_{w_i|v}\Hom(H_{i,w_i},\Q/\Z), 
\end{equation}
which fits in the following commutative diagram:
\begin{equation}\label{eq:rvcd}
\begin{tikzcd}
H^1(k_{v},\wh {T}_{K/k}) \arrow[r,"\iota_{v}^{1}"] \arrow[hook]{d} & \bigoplus_{i=1}^{m}\bigoplus_{w_i\mid v}H^1(K_{i,w_i}, \wh {T}_{K/k}) \arrow[hook]{d} \\
\Hom(G_{v},\Q/\Z)  \arrow[r,"r_{v}"]  & \bigoplus_{i=1}^{m}\bigoplus_{w_i\mid v}\Hom(H_{i,w_i}, \Q/\Z).       
\end{tikzcd}
\end{equation}

We also have the restriction map
\begin{equation}
   r_{(H_i),v}\colon \bigoplus\limits_{i=1}^{m}\Hom(H_i,\Q/\Z) \longrightarrow  \bigoplus\limits_{i=1}^{m}\bigoplus_{w_i|v}\Hom(H_{i,w_i},\Q/\Z),
\end{equation}
which satisfies the following commutative diagram: 
\begin{equation}\label{eq:rhcd}
\begin{tikzcd}
\bigoplus_{i=1}^{m}H^1(K_{i},\wh {T}_{K/k}) \arrow[r,"r_{E/K',v}^{1}"] \arrow[hook]{d} & \bigoplus_{i=1}^{m}\bigoplus_{w_i\mid v}H^1(K_{i,w_i}, \wh {T}_{K/k}) \arrow[hook]{d} \\
\bigoplus_{i=1}^{m}\Hom(H_{i},\Q/\Z)  \arrow[r,"r_{(H_i),v}"]  & \bigoplus_{i=1}^{m}\bigoplus_{w_i\mid v}\Hom(H_{i,w_i}, \Q/\Z).       
\end{tikzcd}
\end{equation}

\begin{defn} (1) Analogous to \eqref{eq:rvcd}, set 
\[ D_v(K,K'):={\rm im} (r_v), \] 
called the subgroup of {\it diagonal} elements of $\bigoplus\limits_{i=1}^{m}\bigoplus_{w_i|v}\Hom(H_{i,w_i},\Q/\Z)$ at $v$. 

    (2) We say an element $x\in \bigoplus\limits_{i=1}^{m}\Hom(H_i,\Q/\Z)$ is {\it locally diagonal at $v$} if the element $r_{(H_i),v}(x)$ lies in the group $D_v(K,K')$ of diagonal elements at $v$; $x$ is said to be {\it locally diagonal everywhere} if it is locally diagonal at $v$ for all places $v$ of $k$. 
    Denote by 
    \[ S_{\rm ld}(K,K')\subset \bigoplus\limits_{i=1}^{m}\Hom(H_i,\Q/\Z) \] 
    the subgroup consisting of all locally diagonal everywhere elements. 
\end{defn}

By the diagrams \eqref{eq:nat_isom}, \eqref{eq:rvcd} and \eqref{eq:rhcd}, the group $S(K,K')$ defined in Definition~\ref{def:SKK'} corresponds to a subgroup of $S_{\rm ld}(K,K')$. 
Moreover, if $K$ is a field which is Galois over $k$, then all the vertical homomorphisms in \eqref{eq:nat_isom}, \eqref{eq:rvcd} and \eqref{eq:rhcd} are isomorphisms by Proposition~\ref{prop:2.4}. Thus, we obtain the following description of $S(K,K')$ and $D(K,K')$. 


\begin{prop}\label{prop:ld}
Under the above notation, the injection
\begin{equation*}
\bigoplus_{i=1}^{m}H^{1}(K_{i},\widehat{T}_{K/k})\hookrightarrow \bigoplus_{i=1}^{m}\Hom(H_{i},\Q/\Z)
\end{equation*}
maps $S(K,K')$ and $D(K,K')$ onto subgroups of $S_{\ld}(K,K')$ and $D_{\diag}(K,K')$ respectively. Moreover, if $K$ is a field which is Galois over $k$, then the homomorphism $S(K,K')\hookrightarrow S_{\ld}(K,K')$ and $D(K,K')\hookrightarrow D_{\diag}(K,K')$ are isomorphisms. 
\end{prop}

We write $L=\prod_{i=0}^m K_i$ and assume $K=K_0$ is a Galois extension of $k$ with group $G=\Gal(K/k)$.
In the following, we shall compare the group $\Sha^2(k,\wh T_{L/k})$ with $\Sha^2(F,\wh T_{L/F})$, where
$F:=\cap_{i=0}^m K_i$. Set $G_F:=\Gal(K/F)$.
Considering $L=K\times K'$ as an \'etale $F$-algebra, we define an $F$-torus $S_{{K,K'}/F}$ as in \eqref{eq:Def of S_{K,K'}}. Similar to \eqref{eq:lg-exact}, we have the commutative diagram
\begin{equation}\label{eq:lg-exact_prime}
 \begin{tikzcd}
H^1(F,\wh{T}_{K/F}) \arrow[r, "\iota_F"] \arrow[d, "r_{K/F}"] & {\bigoplus\limits_{i=1}^{m}H^1(K_i,\wh{T}_{E_i/K_i})} \arrow[d, "r_{E/K'}"]  \\
\prod\limits_u H^1(F_u,\wh{T}_{K_u/F_u}) \arrow[r, "\iota_{F,a}"]         & {\bigoplus\limits_{i=1}^{m}\prod\limits_u H^1(K_{i,u},\wh{T}_{E_{i,u}/K_{i,u}})},          \end{tikzcd}   
\end{equation}
where $u$ runs through the places of $F$, $F_u$ is the completion of $F$ at $u$, and $K_{i,u}:=K_i\otimes_{F} F_u$ (resp.~$E_{i,u}:=E_i\otimes_{F} F_u$).
Set 
\begin{equation}
S_F(K,K'):=\left\{ x\in \bigoplus_{i=1}^{m}H^1(K_i,\wh{T}_{E_i/K_i}): r_{E/K'}(x)\in \im{(\iota_{F,a})} \right\}.    
\end{equation}
and $D_F(K,K'):=\iota_F(H^1(F,\wh{T}_{K/F}))$.

\begin{prop}\label{prop:L/FvsL/k}
   Assume that $G$ is abelian. Then $S_F(K,K')=S(K,K')$ and $D_F(K, K')=D(K,K')$. 
\end{prop}
\begin{proof}
   Set 
   \[ \text{$A:=\bigoplus_{i=1}^{m}H^1(K_i,\wh{T}_{E_i/K_i})$ \quad and \quad $B_{i,\bullet}:=H^1(K_{i,\bullet},\wh{T}_{E_{i,\bullet}/K_{i,\bullet}})$} \] 
   for $\bullet \in \{v, u, w_i\}$, where $v$, $u$, or $w_i$
denotes a place of $k$, $F$, or $K_i$, respectively, with the relation $w_i|u|v$,  
and $K_{i,v}$ (resp.~$K_{i,u}$) denotes $K_i\otimes_k k_v$ (resp.~$K_i\otimes_F F_u$) and so on for $E_{i,v}$ and $E_{i,u}$. The map $\iota^1:H^1(k, \wh T_{K/k})\to A$ factors as 
\[ 
\begin{CD}
     \iota^1 \colon H^1(k, \wh T_{K/k}) @>{b}>> H^1(F,\wh{T}_{K/F}) @>{\iota_F}>> A.  
\end{CD}
\]
Using Lemma~\ref{lm:base_change} and Proposition~\ref{prop:2.4}, the map $b$ can be described as the restriction map
\begin{equation*}
\Hom(G,\Q/\Z)\to \Hom(G_F, \Q/\Z), 
\end{equation*}
which is surjective as $G$ is abelian. It then follows that $D(K,K')=D_F(K,K')$.

To show $S(K,K')=S_F(K,K')$, we need to show that for any $x=(x_i)\in A$ and any place $v$, 
the element \[ r_{E/K',v}(x)\in \bigoplus_{i=1}^m B_{i,v}=\bigoplus_{i=1}^m \prod_{w_i|v} B_{i,w_i}\]
lies in the image of $\iota^1_v$ if and only if it lies in the image of the map 
\[ (\iota_{F,u})_{u|v}\colon \prod_{u|v} H^1(F_u,\wh{T}_{K_u/F_u})\to \bigoplus_{i=1}^m \prod_{u|v} \prod_{w_i|u} B_{i,w_i}=\bigoplus_{i=1}^m \prod_{w_i|v} B_{i,w_i}. \] 
Since the groups $B_{i,w_i}$ for $w_i|v$ are the same and every two $w_i$-components of $r_{E/K}(x)$ are the same, it suffices to check one $w_i$-component. Thus, one reduces to show that the following two maps have the same image 
\[ \iota^1_{w_i|v}\colon H^1(k_v, \wh T_{K_v/k_v}) \longrightarrow \bigoplus_{i=1}^m B_{i,w_i}, \quad  \iota_{F,{w_i|u}}\colon H^1(F_u, \wh T_{K_u/F_u}) \longrightarrow \bigoplus_{i=1}^m B_{i,w_i}.\]
Since the first map factors through a base change map $b_{F_u/k_v}\colon H^1(k_v, \wh T_{K_v/k_v})\to H^1(F_u, \wh T_{K_u/F_u})$ and it is surjective, the maps $\iota^1_{w_i|v}$ and $\iota_{F,{w_i|u}}$ have the same image. This proves the proposition. \qed
\end{proof}

\begin{cor}\label{cor:L/FvsL/k}
  Let the notation and assumptions be as in Proposition~\ref{prop:L/FvsL/k}. Assume further that $K=K_0$ is a cyclic field extension over $k$ and set $F:=\bigcap_{i=0}^m K_i$. Then there is an isomorphism $\Sha(L/k)\simeq \Sha(L/F)$.   
\end{cor}
\begin{proof}
   By \eqref{eq:2.4} and \eqref{prop: ShaT &ShaS}, it suffices to give an isomorphism between $\Sha^1(F,\wh S_{{K,K'}/F})$ and $\Sha^1(k,\wh S_{K,K'})$. Since $K/k$ is cyclic, $K/F$ is also cyclic. Thus, \[ \Sha^1(F,\wh S_{{K,K'}/F})=S_F(K,K')/D_F(K,K')\quad \text{and}\quad \Sha^1(k,\wh{S}_{K,K'})=S(K,K')/D(K,K'). \]
   Then the corollary follows from Proposition~\ref{prop:L/FvsL/k}. \qed
\end{proof}

\begin{remark}    
    (1) In \cite{lee:jpaa2022} Lee gave an explicit description of the group $\Sha^1(k, T_{L/k})$ for which all factors of are cyclic of $p$-power degrees under the condition that $\bigcap_{i=0}^m K_i=k$. Using Corollary~\ref{cor:L/FvsL/k} and replacing $k$ by $F$, one may remove this condition.    
    
    (2) Using a similar argument to that of Corollary~\ref{cor:L/FvsL/k}, one can also show that 
    there is an isomorphism $\Sha^2_\omega(F,\wh{T}_{L/F})\simeq \Sha^2_\omega(k,\wh{T}_{L/k})$.   
\end{remark}

\subsection{Matching $S(K_0,K')$ and $G(K_0,K')$ for the cyclic case} \label{subsec:4.3} 

In this subsection we let $L=\prod_{i=0}^m K_i=K_0\times K'$ and assume that $K_0$ is a cyclic field extension of degree $p^e$. We recall the construction of 
the group $G(K_0,K')$ in \cite{BLP19} and show that this agrees with the group $S(K_0, K')$. 

Let $M_i:=K_0 K_i$, then $M_i$ is a cyclic extension over $K_i$ of degree $p^{e_i}$ for some $e_i\le e$. We may assume $e_1\ge e_2\dots \ge e_m$. Put $I=\{1, \dots, m\}$.

For any nonnegative integers $s\ge t$, write $\pi_{s,t}:\Z/p^s \to \Z/p^t$ the natural projection. For any $n\in \Z/p^{e_1}\Z$ and $a=(a_1,\ldots,a_m)\in \bigoplus_{i=1}^{m} \Z/p^{e_i}\Z$, set 
\begin{equation*}
I_n(a):=\{i\in I: \pi_{e_1,e_i}(n)=a_i\}. 
\end{equation*}
Let 
$\delta(n,a_i)$ be the maximal nonnegative integer $r$ such that $n\equiv a_i \pmod {p^{r}}$. 
For each $v\in \Omega_k$, define  
$e(v)\in \Z$ so that $p^{e(v)}=[K_{0,\wt v}: k_v]$ for a place $\wt v | v$ of $K_0$, and for all $i\in I$, set
\[ e_i(v):=\max\{e_i(w_i): w_i|v \in \Omega_{K_i}\}, \quad p^{e_i(w_i)}=\deg(\wt w_i|w_i):=[M_{i,\wt w_i}: K_{i,w_i}].\] 
It is clear that $e_i(v)\le \min \{e(v),e_i\}$.
Define 
\begin{equation}
  \Omega(I_n(a)) :=\{v\in \Omega_k:\, \forall w|v \in \Omega_{K_i},\  e_i(w)\le \delta(n,a_i),\  \forall\, i\not\in I_n(a)\, \}.  
\end{equation}
Set
\begin{equation}\label{eq:GKK'}
    G(K_0,K'):=\left \{(a_1, a_2,\dots, a_m)\in \bigoplus_{i=1}^{m} \Z/p^{e_i}\Z : \bigcup_{n\in \Z/p^{e_1}\Z} \Omega(I_n(a))=\Omega_k \, \right \}.
\end{equation}

\begin{defn}\label{def:locdiag}
   An element $a=(a_1,a_2, \dots, a_m)\in \bigoplus \Z/p^{e_i}\Z$ is said to be \emph{locally diagonal at $v$} if there exists an element $n\in \Z/p^{e(v)}\Z$ such that $n\equiv a_i\pmod {p^{e_i(v)}}$ for all $i$. It is said to be \emph{locally diagonal everywhere} if it is locally diagonal at $v$ for all $v\in \Omega_k$.  
\end{defn}

\begin{lemma}\label{lm:S=G}
Keep the assumption that $K_0/k$ is a cyclic extension of degree $p^{e}$ and fix an isomorphism
\begin{equation*}
\gamma \colon \Hom(\Gal(K_0/k),\Q/\Z)\xrightarrow{\sim} \Z/p^{e}\Z. 
\end{equation*}
Then the isomorphism $\gamma$ induces an isomorphism
\begin{equation*}
S(K_{0},K')\simeq G(K_{0},K'). 
\end{equation*}
\end{lemma}

\begin{proof}
Since $K_0/k$ is Galois, there is a natural isomorphism $S(K_0,K')\simeq S_{\rm ld}(K_0, K')$ by Proposition \ref{prop:ld}. Therefore, it suffices to show that $\gamma$ induces an isomorphism
\begin{equation*}
S_{\ld}(K_0,K')\simeq G(K_{0},K'). 
\end{equation*}

Note that choosing an isomorphism $\gamma: \Hom(G, \Q/\Z)\isoto \Z/p^e\Z$ is equivalent to choosing a generator of $G$. Indeed, for any generator $\sigma$ of $G$, the corresponding isomorphism is given by  
\[ f_\sigma \colon \Hom(G,\Q/\Z)\isoto \Z/p^e\Z, \quad \chi\mapsto [p^e \chi(\sigma)]. \] 
Here we use the isomorphism $p^e\cdot: (1/p^e) \Z/\Z\isoto \Z/p^e\Z$. Conversely, for a given isomorphism $f\colon \Hom(G,\Q/\Z)\isoto \Z/p^e\Z$, we get an induced isomorphism \[ \Hom(G^\vee, \Q/\Z)=\Hom(G^\vee, (1/p^e) \Z/\Z)\simeq \Hom((1/p^e) \Z/\Z, (1/p^e) \Z/\Z). \] Let $\sigma$ be the generator corresponding to the identity in $\Hom((1/p^e) \Z/\Z, (1/p^e) \Z/\Z)$ by the canonical bijection $G\isoto \Hom(G^\vee, \Q/\Z)$.

If $\gamma=f_\sigma$ for a (unique) generator $\sigma$, then we define an isomorphism $\gamma_{H_i}:\Hom(H_i, \Q/\Z)\isoto \Z/p^{e_i}\Z$ by putting $\gamma_{H_i}:=f_{\sigma_i}$, where $\sigma_i:=(\sigma)^{p^{e-e_i}}$, which is a natural generator of $H_i$. A simple check shows that the following diagram commutes:
\begin{equation*}
\xymatrix@C=35pt{
\Hom(G,\Q/\Z)\ar[d]^{\gamma}\ar[r]^{r}& 
\bigoplus_{i=1}^{m}\Hom(H_{i},\Q/\Z)\ar[d]^{\gamma_{(H_{i})}}\\
\Z/p^{e}\Z \ar[r]^{\pi:=(\pi_{e,e_{i}})_{i}\hspace{5mm}}&\bigoplus_{i=1}^{m}\Z/p^{e_i}\Z. }
\end{equation*}
On the other hand, for a similar reason, we obtain the following commutative diagram 
for each $v\in \Omega_{k}$: 
\begin{equation*}
\begin{tikzcd}
\bigoplus_{i=1}^{m}\Hom(H_{i},\Q/\Z)\arrow[r,"r_{(H_{i}),v}"]\arrow[Isom]{d}& \bigoplus_{i=1}^{m}\bigoplus_{w_{i}\mid v}\Hom(H_{i,w_{i}},\Q/\Z)\arrow[Isom]{d}& 
\Hom(G_{v},\Q/\Z) \arrow[l,"r_{v}"'] \arrow[Isom]{d}\\
\bigoplus_{i=1}^{m}\Z/p^{e_i}\Z \arrow[r,"\bigoplus_{i}\pi_{e_{i},e_{i}(w_{i})}"]&
\bigoplus_{i=1}^{m}\bigoplus_{w_{i}\mid v}\Z/p^{e_{i}(w_{i})}\Z &
\Z/p^{e(v)}\Z. \arrow[l,"\pi_{v}:=(\pi_{e(v),e_{i}(w_{i})})"']
\end{tikzcd}
\end{equation*}
Hence, under the isomorphism
\begin{equation*}
\bigoplus_{i=1}^{m}\gamma_{(H_{i})}\colon \bigoplus_{i=1}^{m}\Hom(H_{i},\Q/\Z)\xrightarrow{\sim} \bigoplus_{i=1}^{m}\Z/p^{e_i}\Z, 
\end{equation*}
the group $S_{\ld}(K,K')$ corresponds to the subgroup $G_{\ld}(K,K')$ of all elements $a\in \bigoplus_{i=1}^{m}\Z/p^{e_i}\Z$ whose images under $\bigoplus_{i}\pi_{e_{i},e_{i}(w_{i})}$ are contained in that of $\pi_{v}$ for each $v\in \Omega_{k}$. Moreover, since $e_{i}(v)$ is defined as the maximum of the $e_{i}(w_{i})$, the group $G_{\ld}(K,K')
\subset \bigoplus_{i=1}^m \Z/{p^{e_i}\Z}$ is the subgroup of all elements which are locally diagonal everywhere in the sense of Definition~\ref{def:locdiag}. 

In the following, we shall prove $G_{\ld}(K,K')=G(K,K')$, which implies the desired assertion. Note that
\textbf{\begin{equation}
\begin{split}
\Omega(I_n(a))& :=\{v\in \Omega_k:\, \forall\, w|v \in \Omega_{K_i},\ e_i(w)\le \delta(n,a_i), \ \forall\, i\not\in I_n(a)\, \} \\
& =\{v\in \Omega_k:\, e_i(v)\le \delta(n,a_i), \ \forall\, i\not\in I_n(a)\, \} \\
&=\{v\in \Omega_k:\, n\equiv a_i\!  \pmod {p^{e_i(v)}},\ \forall\, i\not\in I_n(a) \}.\\
\end{split}
\end{equation}}
\npr The second equality follows from that $e_i(v)=\max\{ e_i(w_i)\}$ for all places $w_i|v$ of $K_i$; while the third equality follows from the definition of $\delta(n,a_i)$.  

Notice that for any $i\in I_n(a)$, we have $n\equiv a_i$ (mod $p^{e_i})$ and hence $n\equiv a_i$ (mod $p^{e_i(v)})$ for all $v\in \Omega_k$. Thus, 
\[a \text{ is locally diagonal at } v\Longleftrightarrow \text{ there exists }n\in \Z/p^{e_1}\Z \text{ such that } v\in\Omega(I_n(a)) \]
and
\[ 
\begin{split}
 \text{$a$ is locally diagonal everywhere} & \Longleftrightarrow \text{for any } v\in\Omega_k, v\in \Omega(I_n(a)) \text{ for some } n\in \Z/p^{e_1}\Z  \\
&\Longleftrightarrow \bigcup_n \Omega(I_n(a))=\Omega_k.   
\end{split}\]
Therefore one has $G_{\ld}(K,K')=G(K,K')$ and this complete the proof. \qed 
\end{proof}

\section{Applications}\label{sec:A}
\subsection{An alternative proof of Theorem~\ref{thm:DW}}

We use the notation above Proposition~\ref{prop:rsqt}. In particular, $G=\Gal(\wt F/k)$ and $H_i=\Gal(\wt FK_i/K_i)$. 
By \eqref{eq:2.4} and \eqref{prop: ShaT &ShaS}, it suffices to prove the vanishing of $\Sha^{1}(k,\widehat{S}_{K,K'})$. By  assumption (3), we have $\widetilde{F}\cap K_{i}=k$ for every $i\in \{1,\ldots,m\}$. Thus, by Proposition~\ref{prop:Sha=S/D} we have the isomorphism 
\begin{equation*}
\Sha^{1}(k,\widehat{S}_{K,K'})\simeq S(K,K')/D(K,K'). 
\end{equation*}
So it remains to show 
the relation $S(K,K')\subset D(K,K')$. 

Since $\widetilde{F}\cap K_{i}=k$, we have $H_i=G$ for each $i\in \{1,\ldots,m\}$.
By Lemma \ref{lem:rsif} for $q=1$, the map $\iota^{1}$ factors as follows: 
\begin{equation*}
H^1(k, \wh {T}_{K/k})\xrightarrow{\diag} \bigoplus_{i=1}^{m}H^1(k, \wh {T}_{K/k})\xrightarrow{\sim}\bigoplus_{i=1}^{m}H^1(K_{i}, \wh {T}_{K/k}). 
\end{equation*}
Combining this with Lemma \ref{lm:base_change}, we obtain a Cartesian diagram
\begin{equation}\label{eq:dwdg}
\begin{tikzcd}
H^1(k, \wh {T}_{K/k}) \arrow[r,"\iota^{1}"] \arrow[hook]{d} & \bigoplus_{i=1}^{m}H^1(K_{i}, \wh {T}_{K/k}) \arrow[hook]{d} \\
\Hom(G,\Q/\Z)  \arrow[r,"\diag"]  & \bigoplus_{i=1}^{m}\Hom(G, \Q/\Z).   
\end{tikzcd}
\end{equation}

In the following, we shall prove the inclusion $S_{\ld}(K,K')\subset D_{\diag}(K,K')$. Note that this implies the desired inclusion $S(K,K')\subset D(K,K')$ as the commutative diagram \eqref{eq:dwdg} is Cartesian. Take $x=(x_{i})_{i}\in S_{\ld}(K,K')$. By  assumption (3), the restriction induces a surjection $\widetilde{H}\twoheadrightarrow G$. Hence it is sufficient to prove $x_{1}(h)=\cdots=x_{m}(h)$ for any $h\in \widetilde{H}$. Now choose $h\in \widetilde{H}$, and pick $\widetilde{w}\in \Omega_{\widetilde{K}}$ whose decomposition group in $\widetilde{G}$ coincides with the cyclic group generated by $h$. Note that such  $\widetilde{w}$ exists by Chebotarev's density theorem. For each $i$, denote by $w_{i,0}$ the unique place of $K_{i}$ which lies below $\widetilde{w}$. Moreover, write $v$ for the restriction of $\widetilde{w}$ to $k$. Then we have the following: 
\begin{equation*}
\xymatrix@C=30pt{
\bigoplus_{i=1}^{m}\Hom(H_{i},\Q/\Z)\ar[r]^{r_{(H_{i}),v}\hspace{8mm}}&
\bigoplus_{i=1}^{m}\bigoplus_{w_{i}\mid v}\Hom(H_{i,w_i},\Q/\Z)\ar[d]^{\bigoplus_{i}\pr_{w_{i,0}}}&
\Hom(G_{v},\Q/\Z) \ar[l]_{\hspace{12mm}r_{v}}\\
&\bigoplus_{i=1}^{m}\Hom(H_{i,w_{i,0}},\Q/\Z). &
}
\end{equation*}
On the other hand, the constructions of $v$ and $w_{i}$ imply that $G_{v}$ and $H_{i,w_{i,0}}$ coincide with the image of $\langle h \rangle$ under the canonical surjection $\widetilde{G}\twoheadrightarrow G$. Hence the above diagram implies that $x=(x_{i})_{i}$ satisfies $x_{1}(h)=\cdots=x_{m}(h)$. Therefore we obtain the desired equality $x_1=\cdots=x_m$. \qed


\subsection{An alternative proof of Theorem~\ref{thm:Pol}}





\begin{lem}\label{lem:ifrs}
Let $k'$ be a finite abelian field extension of $k$. Take a subextension $k''$ of $k'/k$. Then the short exact sequence of $k$-tori
\begin{equation}\label{eq:Tk'k}
1\rightarrow R_{k''/k}T_{k'/k''}\rightarrow T_{k'/k}\xrightarrow{N_{k'/k''}} T_{k''/k}\rightarrow 1
\end{equation}
induces a short exact sequence of abelian groups
\begin{equation}\label{eq:5.3}
0\rightarrow H^{2}(k,\widehat{T}_{k''/k})\xrightarrow{\widehat{N}^2_{k'/k''}} H^{2}(k,\widehat{T}_{k'/k})\xrightarrow{\iota^2_{k'/k''}} H^{2}(k'',\widehat{T}_{k'/k''}). 
\end{equation}
\end{lem}

\begin{proof}
Consider the long exact sequence induced from \eqref{eq:Tk'k}: 
\begin{equation*}
H^{1}(k,\widehat{T}_{k'/k})\xrightarrow{\iota^1_{k'/k''}} H^{1}(k'',\widehat{T}_{k'/k''}) \xrightarrow{\delta^{1}} H^{2}(k,\widehat{T}_{k''/k}) \xrightarrow{\widehat{N}^2_{k'/k''}} H^{2}(k,\widehat{T}_{k'/k})\xrightarrow{\iota^2_{k'/k''}}
H^{2}(k'',\widehat{T}_{k'/k''}) . 
\end{equation*}
By Lemma \ref{lm:base_change} and Proposition \ref{prop:2.4}, there is a commutative diagram
\begin{equation*}
\begin{tikzcd}
H^{1}(k,\widehat{T}_{k'/k})\arrow[r, "\iota^1_{k'/k''}"] \arrow[Isom]{d}& 
H^{1}(k'',\widehat{T}_{k'/k''})\arrow[Isom]{d}\\
\Hom(\Gal(k'/k),\Q/\Z)\arrow["r"]{r}& \Hom(\Gal(k'/k''),\Q/\Z), 
\end{tikzcd}
\end{equation*}
where $r$ is the restriction map. Since $k'/k$ is abelian, the maps $r$ and $\iota^1_{k'/k''}$ are surjective. This shows that  the map $\delta^{1}$ is zero and concludes the short exact sequence \eqref{eq:5.3}. \qed
\end{proof}

\begin{remark}
    The map $\iota^2_{k'/k''}\colon \Sha^2(k, \wh T_{k'/k})\to \Sha^2(k'', \wh T_{k'/k''})$ induced by Lemma \ref{lem:ifrs} is the map obtained from composing the map $\iota^2=b_{k''/k}\colon \Sha^2(k, \wh T_{k'/k})\to \Sha^2(k'', \wh T_{k'\otimes k''/k''})$ with the isomorphism 
    $\Sha^2(k'', \wh T_{k'\otimes k''/k''})\isoto \Sha^2(k'', \wh T_{k'/k''})$ by Corollary~\ref{cor:H1Kn}.
\end{remark}

\begin{lem}\label{lem:bkfc}
Let $K_{0}/k$ be a finite Galois field extension and $K_1/k$ a finite separable field extension. Then there exists an injective homomorphism
\begin{equation*}
\beta \colon \Sha^{2}(K_0\cap K_1,\widehat{T}_{K_0/K_0\cap K_1})\hookrightarrow \Sha^{2}(K_1,\widehat{T}_{K_0/k})
\end{equation*}
such that $b_{K_1/k}=\beta \circ \iota^2_{K_0/{K_0\cap K_1}}:\Sha^2(k, \wh T_{K_0/k})\to \Sha^{2}(K_1,\widehat{T}_{K_0/k})$.
\end{lem}

\begin{proof}
By Proposition~\ref{prop:rsqt} with $K=K_0$ and $K'=K_1$,  
the map $b_{K_{1}/k}$ factors 
as follows: 
\begin{equation*}
b_{K_{1}/k}:\Sha^{2}(k,\widehat{T}_{K_0/k})\xrightarrow{b_{K_0\cap K_1/k}} \Sha^{2}(K_0\cap K_1,\widehat{T}_{K_0/k})\xhookrightarrow{b_{K_1/K_0\cap K_1}} \Sha^{2}(K_1,\widehat{T}_{K_0/k}) 
\end{equation*}
and the map 
$b_{K_1/K_0\cap K_1}$ is injective. 
On the other hand, one has an isomorphism of $(K_0\cap K_1)$-algebras $K_{0}\otimes_{k}(K_0\cap K_1)\simeq \prod_{[K_0\cap K_1:k]}K_0$. Consequently Corollary \ref{cor:H1Kn} gives an isomorphism
\begin{equation*}
\pi' \colon \Sha^{2}(K_0\cap K_1,\widehat{T}_{K_0/k})\xrightarrow{\sim} \Sha^{2}(K_0\cap K_1,\widehat{T}_{K_0/K_0\cap K_1}). 
\end{equation*}
Moreover, we have $\pi'\circ b_{K_0\cap K_1/k}=\iota^2_{K_0/{K_0\cap K_1}}$ by construction. Hence we may define $\beta$ as the composition of $(\pi')^{-1}$ and $b_{K_1/K_0\cap K_1}$. \qed
\end{proof}

\begin{lem}\label{lem:isod}
Let $L=K_0\times K_1$, where $K_0$ and $K_1$ are finite Galois field extensions of $k$. For each $i\in \{0,1\}$, let 
\begin{equation*}
    j_{i}\colon T_{K_{i}/k}\rightarrow T_{L/k}
\end{equation*}
be the homomorphism induced by the natural inclusion $T^{K_{i}}\hookrightarrow T^{L}$. Then one has a commutative diagram
\begin{equation*}
    \begin{tikzcd}
        \Sha^{2}(k,\wh T_{L/k}) \arrow[r,"\wh j_{0}^{2}"] \arrow[d,"\wh j_{1}^{2}"]& 
        \Sha^{2}(k,\wh T_{K_{0}/k})\arrow[d,"\wh N_{K_{0}K_{1}/K_{0}}^{2}"]\\
        \Sha^{2}(k,\wh T_{K_{1}/k})\arrow[r,"\wh N_{K_{0}K_{1}/K_{1}}^{2}"]&
        \Sha^{2}(k,\wh T_{K_{0}K_{1}/k}). 
    \end{tikzcd}
\end{equation*}
\end{lem}

\begin{proof}
This is inspired by the proof of \cite[Lemma 3]{pollio}. Put $G:=\Gal(K_{0}K_{1}/k)$, and consider a commutative diagram
\begin{equation}\label{eq:smtw}
    \begin{tikzcd}
        \Sha^{2}(k,\wh T_{L/k}) \arrow[r,"\delta^{2}"]\arrow[d,"\wh j^{2}"]& H^{3}(G,\Z)\arrow[d,"\mu"]\\
        \Sha^{2}(k,\wh T_{K_{0}/k})\oplus \Sha^{2}(k,\wh T_{K_{1}/k}) \arrow[r,"\delta^{2}\oplus \delta^{2}"]\arrow[d,"\wh N^{2}"]& H^{3}(G,\Z)\oplus H^{3}(G,\Z)\arrow[d,"+"]\\
        \Sha^{2}(k,\wh T_{K_{0}K_{1}/k}) \arrow[r,"\delta^{2}"]& H^{3}(G,\Z). 
    \end{tikzcd}
\end{equation}
Here the homomorphisms in \eqref{eq:smtw} are defined as follows: 
\begin{gather*}
    \wh j^{2}\colon \Sha^{2}(k,\wh T_{L/k})\rightarrow \Sha^{2}(k,\wh T_{K_{0}/k})\oplus \Sha^{2}(G,\wh T_{K_{1}/k});\,c \mapsto (\wh j_{0}^{2}(c),-\wh j_{1}^{2}(c)); \\
    \wh N^{2}\colon \Sha^{2}(k,\wh T_{K_{0}/k})\oplus \Sha^{2}(k,\wh T_{K_{1}/k}) \rightarrow \Sha^{2}(k,\wh T_{K_{0}K_{1}/k}); (c_{0},c_{1})\mapsto \wh N_{K_{0}K_{1}/K_{0}}(c_{0})+\wh N_{K_{0}K_{1}/K_{1}}(c_{1});\\
    \mu \colon H^{3}(G,\Z) \rightarrow H^{3}(G,\Z)\oplus H^{3}(G,\Z);\,x \mapsto (x,-x); \\
    +\colon H^{3}(G,\Z)\oplus H^{3}(G,\Z) \rightarrow H^{3}(G,\Z);\,(x,y)\mapsto x+y. 
\end{gather*}
Then the lowest horizontal map is injective by Tate's theorem. On the other hand, the composite of right vertical maps is zero. Hence $\wh N^{2}\circ \wh j^{2}$ is also zero, which implies the desired assertion $\wh N_{K_{0}K_{1}/K_{0}}^{2}\circ \wh j_{0}^{2}= \wh N_{K_{0}K_{1}/K_{1}}^{2}\circ \wh j_{1}^{2}$. \qed
\end{proof}

\begin{lem}\label{lem:ifcr}
Let $K/k$ be a finite Galois field extension with group $G$. Take two Galois subextensions $K'_{0}\subset K_{1}$ of $K/k$, and set $N:=\Gal(K/K_{1})$ and $N':=\Gal(K/K'_{0})$. Then there is a commutative diagram
    \begin{equation*}
    \begin{tikzcd}
 0  \arrow[r]  & H^{2}(G/N,\wh T_{K'_{0}/k}) \arrow[r, "\Inf"] \arrow[d, "\wh N_{K_{1}/K'_{0}}" ] & H^{2}(G,\wh {T}_{K'_{0}/k}) \arrow[r,"\Res"] \arrow[d, "\wh N_{K_{1}/K'_{0}}"] &
 H^{2}(N,\wh T_{K'_{0}/k})  \arrow[d, "\wh N_{K_{1}/K'_{0}}"]   \\
0 \arrow[r] & H^{2}(G/N,\wh T_{K_{1}/k}) \arrow[r, "\Inf"] & H^{2}(G,\wh {T}_{K_{1}/k}) \arrow[r,"\Res"] &H^{2}(N,\wh T_{K_{1}/k}),                     
\end{tikzcd}
    \end{equation*}
    where the horizontal sequences are exact, and the left square is Cartesian. 
\end{lem}

\begin{proof}
    The exactness of horizontal sequence is a consequence of the triviality of $H^{1}(N,\wh T_{K'_{0}/k})$ and $H^{1}(N,\wh T_{K_{1}/k})$. On the other hand, we have $H^{1}(N,\Ind_{N}^{G}\wh T_{K_{1}/K'_{0}})=0$ since the action of $N$ on $\Ind_{N'}^{G}\wh T_{K_{1}/K'_{0}}$ is trivial. Hence the rightmost $\wh N_{K_{1}/K'_{0}}$ is injective. Using this injectivity, we obtain the remaining assertion by diagram chasing. \qed
\end{proof}

\begin{lem}\label{lem:nmmn}
Let $L=K_0\times K_1$, where $K_0$ and $K_1$ are finite abelian field extensions of $k$. Then the homomorphism of $k$-tori
\begin{equation*}
\nu \colon T_{L/k}\rightarrow T_{K_0\cap K_1/k}, \  (x,y)\mapsto N_{K_0/K_0\cap K_1}(x)N_{K_1/K_0\cap K_1}(y)
\end{equation*}
induces an injection $\wh \nu^2:\Sha^{2}(k,\widehat{T}_{K_0\cap K_1/k}) \hookrightarrow \Sha^{2}(k,\widehat{T}_{L/k})$. 
\end{lem}

\begin{proof}
This is inspired by the proof of \cite[Theorem 1]{pollio}. Consider the homomorphism
\begin{equation*}
\nu_1\colon T_{K_0K_1/k}\rightarrow T_{L/k}, \ x\mapsto (N_{K_0K_1/K_0}(x),1). 
\end{equation*}
By definition, one has $\nu \circ \nu_1=N_{K_0K_1/K_0\cap K_1}$  Therefore the homomorphism $\widehat{N}_{K_0K_1/K_0\cap K_1}^{2}=\wh {\nu_1}^2\circ \wh \nu^2$ factors as follows: 
\begin{equation*}
\widehat{N}_{K_0K_1/K_0\cap K_1}^{2}: \Sha^{2}(k,\widehat{T}_{K_0\cap K_1/k})\xrightarrow{\widehat{\nu}^{2}} \Sha^{2}(k,\widehat{T}_{L/k})\rightarrow \Sha^{2}(k,\widehat{T}_{K_0K_1/k}). 
\end{equation*}
By Lemma \ref{lem:ifrs}, the map $\widehat{N}_{K_0K_1/K_0\cap K_1}^{2}$ is injective and so is $\wh \nu^2$. \qed
\end{proof}



\npr {\bf Proof of Theorem~\ref{thm:Pol}}
By Lemma~\ref{lem:nmmn}, to prove $\wh \nu^2:\Sha^{2}(k,\widehat{T}_{K_0\cap K_1/k}) \to \Sha^{2}(k,\widehat{T}_{L/k})\simeq \Sha^1(k,\wh S_{K_0,K_1})$ is an isomorphism, it is sufficient to construct an injection 
\begin{equation}\label{eq:inj}
\Sha^{1}(k,\widehat{S}_{K_0,K_1})\hookrightarrow \Sha^{2}(k,\widehat{T}_{K_0\cap K_1/k}). 
\end{equation}

Consider the exact sequence in Proposition 4.3. Then we have $D(K_0,K_1)=H^{1}(K_1,\widehat{T}_{E_1/K_1})$ since $K_0K_1/k$ is abelian. In particular, the group $S(K_0,K_1)/D(K_0,K_1)$ is trivial, and the homomorphism
\begin{equation*}
\delta^{1}\colon H^{1}(k,\widehat{S}_{K_0,K_1})\rightarrow H^{2}(k,\widehat{T}_{K_0/k})
\end{equation*}
is injective. So we arrive at the following diagram
\begin{equation}\label{eq:5.5}
    \begin{tikzcd}
    &  &  H^{1}(k,\widehat{S}_{K_0,K_1}) \arrow[d, hook, "\delta^1"] & \\
   0\arrow[r] & H^{2}(k,\widehat{T}_{K_0\cap K_1/k}) \arrow[r,"\widehat{N}^2_{K_0/K_0\cap K_1}"] & H^{2}(k,\widehat{T}_{K_0/k}) \arrow[r, "\iota^2_{K_0/K_0\cap K_1}"] & H^{2}(K_0\cap K_1,\widehat{T}_{K_0/K_0\cap K_1})
\end{tikzcd}
\end{equation}
by Lemma~\ref{lem:ifrs}.

Since the sequence
\begin{equation*}
H^{1}(k,\widehat{S}_{K_0,K_1})\xrightarrow{\delta^1} H^{2}(k,\widehat{T}_{K_0/k})\xrightarrow{b_{K_1/k}} H^{2}(K_1,\widehat{T}_{K_0\otimes K_1/K_1})
\end{equation*}
is exact, the image of $\Sha^{1}(k,\widehat{S}_{K_0,K_1})$ under $\delta^{1}$ is contained in the kernel of $b_{K_1/k}$. 
Lemma~\ref{lem:bkfc} gives the factorization
\[ \begin{tikzcd}
b_{K_1/k}\colon 
\Sha^2(k, \wh T_{K_0/k}) \arrow[r, "\iota^2_{K_0/{K_0\cap K_1}}"] 
& 
\Sha^{2}(K_0\cap K_1,\widehat{T}_{K_0/K_0\cap K_1}) \arrow[r, hook, "\beta"] &
  \Sha^{2}(K_1,\widehat{T}_{K_0/k})=\Sha^{2}(K_1,\widehat{T}_{K_0\otimes K_1/K_1}).   
\end{tikzcd} \]
It follows that $\Ker(b_{K_1/k})\cap \Sha^{2}(k,{\wh T}_{K_0/k})=\Ker(\iota^2_{K_0/K_0\cap K_1})\cap \Sha^{2}(k,{\wh T}_{K_0/k})$.
Therefore, the map $\delta^1 \colon \Sha^1(k,\wh S_{K_0,K_1}) \hookrightarrow \Sha^2(k,\wh T_{K_0/k})$ factors through an injective map
\begin{equation}\label{eq:gamma}
\gamma \colon \Sha^{1}(k,\widehat{S}_{K_0,K_1})\hookrightarrow H^{2}(k,\widehat{T}_{K_0\cap K_1/k}). 
\end{equation}


We now show that the image of $\Sha^{1}(k,\widehat{S}_{K_0,K_1})$ under $\gamma$ is contained in $\Sha^{2}(k,\widehat{T}_{K_0\cap K_1/k})$. Put $G:=\Gal(K_{0}K_{1}/k)$, $H_{0}:=\Gal(K_{0}K_{1}/K_{0})$ and $H_{1}:=\Gal(K_{0}K_{1}/K_{1})$. Then we have $H^{1}(k,\wh S_{K_{0},K_{1}})=H^{1}(G,\wh S_{K_{0},K_{1}})$. Moreover, the image of $\Sha^{1}(k,\wh S_{K_{0},K_{1}})$ under $\delta^{1}$ is contained in the kernel of $\Res \colon H^{2}(G,\wh T_{K_{0}/k})\rightarrow H^{2}(H_{1},\wh T_{K_{0}/k})$. This implies that the restriction of $\delta^{1}$ to $\Sha^{1}(k,\wh S_{K_{0},K_{1}})$ factors through $H^{2}(G,\wh T_{K_{0}\cap K_{1}}/k)$. On the other hand, it also factors through $H^{2}(G/H_{0},\wh T_{K_{0}/k})$ since $\Sha^{2}(k,\wh T_{K_{0}/k})\subset H^{2}(G/H_{0},\wh T_{K_{0}/k})$. Therefore, by Lemma \ref{lem:ifcr}, we obtain a factorization of $\delta^{1}$ as follows: 
\begin{equation*}
    \Sha^{1}(k,\wh S_{K_{0},K_{1}}) \xrightarrow{\delta_{0}^{1}}H^{2}(G/H_{0},\wh T_{K_{0}\cap K_{1}/k}) \hookrightarrow H^{2}(G,\wh T_{K_{0}/k}). 
\end{equation*}
In the following, we shall prove that the composite
\begin{equation}\label{eq:rshi}
    \Sha^{1}(k,\wh S_{K_{0},K_{1}})\xrightarrow{\delta_{0}^{1}} H^{2}(G/H_{0},\widehat{T}_{K_{0}\cap K_{1}/k})\xrightarrow{\Res} H^{2}(H_{1}H_{0}/H_{0},\widehat{T}_{K_{0}\cap K_{1}/k})
\end{equation}
is zero. Consider a commutative diagram
\begin{equation}\label{eq:ifrd}
\begin{tikzcd}
 H^{2}(G/H_{0}, \wh T_{K_{0}\cap K_{1}/k}) \arrow[r,"\Inf"] \arrow[d, "\Res" ] & H^{2}(G,\wh T_{K_{0}\cap K_{1}/k}) \arrow[r,"\wh N_{K_{1}/K_{0}\cap K_{1}}"] \arrow[d, "\Res"] & H^{2}(G,\wh T_{K_{1}/k}) \arrow[d,"\Res"]  \\
 H^{2}(H_{1}H_{0}/H_{0}, \wh T_{K_{0}\cap K_{1}/k}) \arrow[r,"\Inf"] \arrow[Isom]{d} & H^{2}(H_{0}H_{1},\wh T_{K_{0}\cap K_{1}/k}) \arrow[r,"\wh N_{K_{1}/K_{0}\cap K_{1}}"] \arrow[d, "\Res"] & H^{2}(H_{0}H_{1},\wh T_{K_{1}/k}) \arrow[d,"\Res"]  \\
 H^{2}(H_{1}/H_{0}\cap H_{1}, \wh T_{K_{0}\cap K_{1}/k}) \arrow[r, "\Inf"] & H^{2}(H_{1},\wh T_{K_{0}\cap K_{1}/k}) \arrow[r,"\wh N_{K_{1}/K_{0}\cap K_{1}}"] & H^{2}(H_{1},\wh T_{K_{1}/k}).            
\end{tikzcd}
\end{equation}
First, we claim that the composite of $\delta_{0}^{1}$ and the homomorphism $H^{2}(G/H_{0},\wh T_{K_{0}\cap K_{1}/k})\rightarrow H^{2}(H_{1},\wh T_{K_{1}/k})$ defined by \eqref{eq:ifrd} is zero. By construction, the map $\delta^{1}\colon \Sha^{1}(k,\wh S_{K_{0},K_{1}})\rightarrow \Sha^{2}(k,\wh T_{K_{0}/k})$ factors as follows
\begin{equation*}
    \Sha^{1}(k,\wh S_{K_{0},K_{1}})\xrightarrow{\Delta^{1}} \Sha^{2}(k,\wh T_{L/k})\xrightarrow{\wh j_{0}} \Sha^{2}(k,\wh T_{K_{0}/k}),
\end{equation*}
where $\wh j_{0}$ is as in Lemma \ref{lem:isod}. Recall that Lemma \ref{lem:ifrs} implies that $\wh N_{K_{0}K_{1}/K_{0}}^{2}$ and $\wh N_{K_{0}K_{1}/K_{1}}^{2}$ (in Lemma~\ref{lem:isod}) are injective since $K_{0}$ and $K_{1}$ are abelian over $k$. Hence the composite
\begin{equation*}
    \Sha^{1}(k,\wh S_{K_{0},K_{1}})\xrightarrow{\Delta^{1}} \Sha^{2}(k,\wh T_{L/k})\xrightarrow{\wh j_{1}} \Sha^{2}(k,\wh T_{K_{1}/k})\hookrightarrow H^{2}(G,\wh T_{K_{1}/k})
\end{equation*}
coincides with the composite
\begin{equation*}
    \Sha^{1}(k,\wh S_{K_{0},K_{1}})\xrightarrow{\delta_{0}^{1}} H^{2}(G/H_{0},\wh T_{K_{0}\cap K_{1}/k})\xrightarrow{\Inf} H^{2}(G,\wh T_{K_{0}\cap K_{1}/k}) \xrightarrow{\wh N_{K_{1}/K_{0}\cap K_{1}}} H^{2}(G,\wh T_{K_{1}/k}), 
\end{equation*}
the composition map of the first row in \eqref{eq:ifrd}.
Hence Lemma \ref{lem:ifcr} implies that the composite of $\delta_{0}^{1}$ and $H^{2}(G/H_{0},\wh T_{K_{0}\cap K_{1}/k})\rightarrow H^{2}(H_{1},\wh T_{K_{1}/k})$ is zero as desired. 

We return to prove that the composite $\Res\circ \delta^1_0$ in \eqref{eq:rshi} is zero. By Lemma \ref{lem:ifcr}, one has the injectivity of the lowest $\Inf$ in \eqref{eq:ifrd}. The same property holds for the lowest $\wh N_{K_{1}/K_{0}\cap K_{1}}$ since the action of $H_{1}$ on $\Ind_{H_{0}H_{1}}^{G}\wh T_{K_{1}/K_{0}\cap K_{1}}$ is trivial. Consequently we get the desired assertion. Since $\Res\circ \delta^1_0=0$, we conclude a factorization of $\delta_{0}^{1}$ as follows: 
\begin{equation*}
    \Sha^{1}(k,\wh S_{K_{0},K_{1}})\xrightarrow{\delta_{0,1}^{1}} H^{2}(G/H_{0}H_{1},\widehat{T}_{K_{0}\cap K_{1}/k})\hookrightarrow H^{2}(G/H_{0},\widehat{T}_{K_{0}\cap K_{1}/k}). 
\end{equation*}
Now we consider the image of $\delta_{0,1}^{1}$ and apply the local-global maps. Take any place $v$ of $k$. Then one has a commutative diagram
\begin{equation*}
    \begin{tikzcd}
 H^{2}(G_{v}H_{0}H_{1}/H_{0}H_{1},\wh {T}_{K_{0}\cap K_{1}/k}^{2}) \arrow[r,"\wh N_{K_{0}/K'_{0}}"] \arrow[d, "\delta^{2}"] &
 H^{2}(G_{v}H_{0}/H_{0},\wh T_{K'_{0}/k})  \arrow[d, "\delta^{2}"]   \\
H^{3}(G_{v}H_{0}H_{1}/H_{0}H_{1},\Z) \arrow[r,"\Inf"] &H^{3}(G_{v}H_{0}/H_{0},\Z). 
\end{tikzcd}
\end{equation*}
However, the lower horizontal map is injective since the homomorphism $H^{2}(G_{v}H_{0}/H_{0},\Z)\rightarrow H^{2}((G_{v}H_{0}\cap H_{0}H_{1})/H_{0},\Z)$ is surjective. Therefore the image of $\delta_{0,1}^{1}$ is contained in $\Sha^{2}(k,\wh T_{K_{0}\cap K_{1}/k})$. This completes the proof of Theorem \ref{thm:Pol}. \qed

\section{Reduction to subextensions of $p$-power degree}
    We keep the notation of Section~\ref{sec:C}.
    For any abelian extension $M$ over $k$ and any prime number $p$, let $M(p)$ denote the maximal subextension of $M$ over $k$ of $p$-power degree. If $G$ is an abelian group, write $G(p)$ for the $p$-primary torsion subgroup of $G$. 
    
    When $L=K\times K'$ and $K$ is cyclic of degree $d$ over $k$, one has the following decomposition result  \cite[Proposition 5.16]{BLP19}:
 \begin{equation}
     \label{eq:ShaL}
     \Sha(L/k) \simeq \bigoplus_{p|d} \Sha(L[p]/k), \quad \text{where } L[p]=K(p)\times K'. 
 \end{equation} 
    This reduces the computation of $\Sha(L/k)$ to that of $\Sha(L[p]/k)$.
    In this section, we extend this result to the case where the field extension $K/k$ does not need to be cyclic.

\begin{lem}\label{H^3(G,Z)}
Let $K$ be a Galois field extension over $k$ and $M$ be a Galois subextension. Write $G=\Gal(K/k)$ and $\bar G=\Gal(M/k)$ for the Galois groups of $K$ and $M$ over $k$, respectively. Then we have the commutative diagram:
\begin{equation}\label{eq:H^2K/k(p)}
\begin{tikzcd}
 H^i(\bar G, \wh T_{M/k}) \arrow[r, "\sim"] \arrow[d, , "\wh N_{K/M}"] & H^{i+1}(\bar G,\Z) \arrow[d, "\Inf"] \\
H^i(G, \wh T_{K/k}) \arrow[r, "\sim"]  & H^{i+1}(G,\Z),
\end{tikzcd}    
\end{equation}
for any $i\ge 1$, where $N_{K/M}:T_{K/k}\to T_{M/k}$ is the morphism induced from the natural norm map $N_{K/M}: T^K\to T^{M}$.
\end{lem}

\begin{proof}
The natural norm map $N_{K/M}: T^K\to T^{M}$ gives the following commutative diagram of $k$-tori:
\begin{center}
\begin{tikzcd}
 1  \arrow[r]  & T_{K/k} \arrow[r] \arrow[d, "N_{K/M}"] & T^{K} \arrow[r,  "N_{K/k}"] \arrow[d, "N_{K/M}"] & \G_m \arrow[r] \arrow[d, "\id" ] & 1  \\
1 \arrow[r] & T_{M/k}  \arrow[r]                    & T^{M} \arrow[r, "N_{M/k}"]                    & \G_m \arrow[r]                    & 1.                     
\end{tikzcd}
\end{center}

Taking the dual, we get 
\begin{equation}\label{dual}
\begin{tikzcd}
 0  \arrow[r]  & \Z \arrow[r, "\wh N_{M/k}"] \arrow[d, "\id" ] & \wh {T}^{M}=\Ind_{1}^{\bar G}\Z \arrow[r] \arrow[d, "\wh N_{K/M}"] &\wh T_{M/k} \arrow[r] \arrow[d, "\wh N_{K/M}"]  & 0  \\
0 \arrow[r] & \Z \arrow[r, "\wh N_{K/k}"]                 & \wh {T}^{K}=\Ind_{1}^{G}\Z \arrow[r]                    &\wh T_{K/k}  \arrow[r]                        & 0.                     
\end{tikzcd}
\end{equation}

Taking cohomology and using inflation, we obtain the following relations 
\begin{center}
\begin{tikzcd}
 0  \arrow[r] \arrow[d]  & H^i(\bar G, \wh T_{M/k}) \arrow[r, "\sim"] \arrow[d, "\Inf" ] & H^{i+1}(\bar G,\Z) \arrow[r] \arrow[d, "\Inf"] & 0 \arrow[d]  \\
 H^i(G,\Ind_1^{\bar G}\Z)  \arrow[r] \arrow[d]  & H^i(G, \wh T_{M/k}) \arrow[r] \arrow[d, "\wh N_{K/M}" ] & H^{i+1}(G,\Z) \arrow[r] \arrow[d, "\id"] & H^{i+1}(G,\Ind_1^{\bar G}\Z) \arrow[d]  \\
 0  \arrow[r]  & H^i(G, \wh T_{K/k}) \arrow[r, "\sim"] & H^{i+1}(G,\Z) \arrow[r] & 0.            
\end{tikzcd}
\end{center}

So we get the commutative diagram:
\begin{center}
\begin{tikzcd}
 H^i(\bar G, \wh T_{M/k}) \arrow[r, "\sim"] \arrow[d, , "\wh N_{K/M}"] & H^{i+1}(\bar G,\Z) \arrow[d, "\Inf"] \\
H^i(G, \wh T_{K/k}) \arrow[r, "\sim"]  & H^{i+1}(G,\Z). \qed
\end{tikzcd}   
\end{center}
\end{proof}


In the rest of this section, we assume that $K$ is a Galois extension of $k$ with Galois group $G$ of the form $G=H\rtimes G_p$, where $G_p$ is a Sylow $p$-subgroup and $H$ is a normal subgroup of prime-to-$p$ order. 

Let $K(p)$ be the maximal $p$-power subextension of $K/k$ fixed by $H$ and $G'_p:=\Gal(K(p)/k)$ be the Galois group of $K(p)$ over $k$. Denote $G'_{p,v}$ as the decomposition group of a place $v$ in $G'_p$. Consider the natural projection $\pi: G\to G'_p$; we have $\pi(H)=0$ and that $\pi:G_p\isoto G'_p$ is an isomorphism.
Let $N_{K/K(p)}: T_{K/k} \to T_{K(p)/k}$ be the morphism induced by the norm map $N_{K/K(p)}:T^K \to T^{K(p)}$.

\begin{lem}\label{lem: decomp_H^1(T) Norm}
Let $K$ and $G=H \rtimes G_p$ be as above. Then the map $\wh N_{K/K(p)}:H^{1}(k,\wh T_{K(p)/k})\isoto H^{1}(k,\wh T_{K/k})(p)$ is an isomorphism. 
\end{lem}

\begin{proof}
Note that $H=\Gal(K/K(p))$ and we have 
$\wh {T}^{K(p)}=\Ind_{H}^G \Z$. 
Taking  cohomology of (\ref{dual}) and using inflation, we obtain a similar commutative diagram as \eqref{eq:H^2K/k(p)}:
\begin{center}
\begin{tikzcd}
H^1(G'_p,\wh T_{K(p)/k}) \arrow[r, "\sim"] \arrow[d, "\wh N_{K/K(p)}"] 
& \Hom(G'_p, \Q/\Z) \arrow[d, hook, "\Inf"] \\
H^1(G,\wh T_{K/k}) \arrow[r, "\sim"] 
& \Hom(G,\Q/\Z).  
\end{tikzcd}    
\end{center}
Note $\Hom(G,\Q/\Z)(p)=\Hom(G/H,\Q/\Z)=\Hom(G'_p, \Q/\Z)$. Thus, we get the isomorphism $\wh N_{K/K(p)}:H^{1}(k,\wh T_{K(p)/k})\isoto H^{1}(k,\wh T_{K/k})(p)$. \qed
\end{proof}

\begin{lem}\label{lm:GpxH}
   Let $G$ be a finite group of the form $G=H\rtimes G_p$, where $G_p$ is a Sylow $p$-subgroup and $H$ is a normal subgroup of prime-to-$p$ order. 
   Then the inflation map $\Inf: H^i(G_p, \Z)\to H^i(G,\Z)(p)$ is an isomorphism for $i\ge 0$.   
\end{lem}
\begin{proof}
   The inclusion $\iota:G_p\embed G$ induces the restriction $\Res:H^i(G,\Z) \to H^i(G_p,\Z)$, which induces an injective map 
   $\Res: H^i(G,\Z)(p) \to H^i(G_p,\Z)$ by \cite[Theorem 4, p.~140]{Serre_local}. The inflation map $\Inf: H^i(G_p,\Z) \to H^i(G,\Z)$ is induced by the natural projection $\pr: G\to G_p$. As the composition $G_p\to G\to G_p$ is the identity, so is the composition  
   \[ \begin{CD} H^i(G_p,\Z) @>\Inf >> H^i(G,\Z)(p) @>\Res>> H^i(G_p,\Z). \end{CD} \]
 It follows that $\Res: H^i(G,\Z)(p) \to H^i(G_p,\Z)$ is surjective and hence an isomorphism. Therefore, $\Inf: H^i(G_p,\Z) \to H^i(G,\Z)(p)$ is an isomorphism. \qed 
\end{proof}

\begin{cor}\label{wh N^2}
    Let $K$ and $G=H\rtimes G_p$ be as above. Then the map $\wh N_{K/K(p)}^2: \Sha^2(k,\wh T_{K(p)/k})\to \Sha^2(k, \wh T_{K/k})(p)$ is an isomorphism. 
\end{cor}

\begin{proof}
    Since $G_p \simeq G'_p$, Lemma~\ref{lm:GpxH} shows that the inflation map $\Inf: H^3(G'_p, \Z)\to H^3(G,\Z)(p)$ is an isomorphism. 
Note that the natural projection $\pi$ maps $G_v$ onto $G'_{p,v}$, and we have the following commutative diagram:
\begin{center}
    \begin{tikzcd}
    G  \arrow[r, twoheadrightarrow, "\pi" ]    &  G'_p \\
    G_v \arrow[r, twoheadrightarrow, "\pi" ] \arrow[u,hook]   &  G'_{p,v} \arrow[u,hook].
    \end{tikzcd}
\end{center}
This gives us the commutative diagram:
\begin{center}
\begin{tikzcd}
H^3(G'_p,\Z) \arrow[r, "\sim"', "\Inf"] \arrow[d] 
& H^3(G,\Z)(p) \arrow[d] \\
H^3(G'_{p,v},\Z) \arrow[r, "\sim"', "\Inf"] 
& H^3(G_v,\Z)(p).
\end{tikzcd}    
\end{center}

Therefore, we obtain an isomorphism $\Sha^3(G'_p,\Z)\isoto \Sha^3(G,\Z)(p)$, and hence by Tate's theorem the map $\wh N_{K/K(p)}^2: \Sha^2(k,\wh T_{K(p)/k})\isoto \Sha^2(k, \wh T_{K/k})(p)$ is an isomorphism. \qed

\end{proof}

\begin{thm}\label{thm:G(p)xH}
Assume that $K$ is a Galois extension over $k$ with Galois group $G$ of the form $G=H\rtimes G_p$, where $H\lhd G$ is a normal subgroup with prime-to-$p$ order and $G_p$ is a Sylow $p$-subgroup. Let $K(p)$ be the subextension of $K/k$ fixed by $H$ and let $L[p]=K(p)\times K'$. Then there is a natural isomorphism $\Sha(L/k)(p)\isoto \Sha(L[p]/k)$. 
\end{thm}

\begin{proof}
By replacing $E$ and $K$ by $E(p)=K(p)\otimes K'$ and $K(p)$ in the short exact sequence \eqref{eq:S_{K,K'}}, respectively, and using the relation $R_{K'/k}(T_{E(p)/K'})=\prod\limits_{i=1}^m R_{K_i/k}(T_{E_i(p)/K_i})$ in $T^{E(p)}$, 
we have 
\[ \begin{tikzcd}
1 \arrow[r] & S_{K(p),K'}  \arrow[r]                    & \prod\limits_{i=1}^{m} R_{K_i/k}(T_{E_i(p)/K_i}) \arrow[r, "N_{E(p)/K(p)}"]                    & T_{K(p)/k} \arrow[r]                    & 1.                     
\end{tikzcd} \]
We have the following commutative diagram of $k$-tori through natural norm maps:  
\begin{center}
\begin{tikzcd}
 1  \arrow[r]  & S_{K,K'} \arrow[r] \arrow[d, "N_{E/E(p)}"] & \prod\limits_{i=1}^{m} R_{K_i/k}(T_{E_i/K_i}) \arrow[r,  "N_{E/K}"] \arrow[d, "N_{E/E(p)}"] & T_{K/k}\arrow[r] \arrow[d, "N_{K/K(p)}" ] & 1  \\
1 \arrow[r] & S_{K(p),K'}  \arrow[r]                    & \prod\limits_{i=1}^{m} R_{K_i/k}(T_{E_i(p)/K_i}) \arrow[r, "N_{E(p)/K(p)}"]                    & T_{K(p)/k} \arrow[r]                    & 1.                     
\end{tikzcd}
\end{center}

Taking the dual, we obtain the following commutative diagram of $\Gamma_k$-modules
\begin{center}
\begin{tikzcd}
 0  \arrow[r]  & \wh T_{K(p)/k}\arrow[r,  "\wh N_{E(p)/K(p)}"] \arrow[d, "\wh N_{K/K(p)}"] & \bigoplus\limits_{i=1}^{m}\Ind_{\Gamma_{K_i}}^{\Gamma_k}(\wh{T}_{E_i(p)/K_i}) \arrow[r] \arrow[d, "\wh N_{E/E(p)}"] & \wh S_{K(p),K'} \arrow[r] \arrow[d, "\wh N_{E/E(p)}" ] & 0  \\
0 \arrow[r] & \wh T_{K/k}  \arrow[r, "\wh N_{E/K}"]                    & \bigoplus\limits_{i=1}^{m}\Ind_{\Gamma_{K_i}}^{\Gamma_k}(\wh{T}_{E_i/K_i}) \arrow[r]                    & \wh S_{K,K'} \arrow[r]                    & 0.                      
\end{tikzcd}
\end{center}

This yields the following commutative diagram
\begin{footnotesize}
\begin{equation}\label{eq:five lemma (ab)}
 \begin{tikzcd}
  H^1(k,\wh T_{K(p)/k}) \arrow[r, "\iota(p)"] \arrow[d, "\wh N_{K/K(p)}"] & \bigoplus\limits_{i=1}^{m}H^1(K_i, \wh{T}_{E_i(p)/K_i}) \arrow[r, "\rho(p)"] \arrow[d, "\wh N_{E/E(p)}"] & H^1(k, \wh S_{K(p),K'}) \arrow[r, "\delta(p)"] \arrow[d, "\wh N_{E/E(p),S}" ] & H^2(k,\wh T_{K(p)/k}) \arrow[d, "\wh N_{K/K(p)}^2" ] \arrow[r, "\iota(p)^2"] & {\bigoplus\limits_{i=1}^{m}H^2(K_i,\wh{T}_{E_i(p)/K_i})}  \arrow[d, "\wh N_{E/E(p)}^2"]\\
  H^1(k,\wh T_{K/k})  \arrow[r, "\iota"]                    & \bigoplus\limits_{i=1}^{m}H^1(K_i, \wh{T}_{E_i/K_i}) \arrow[r, "\rho"]                    & H^1(k, \wh S_{K,K'}) \arrow[r, "\delta"]  &  H^2(k,\wh T_{K/k}) \arrow[r, "\iota^2"] &{\bigoplus\limits_{i=1}^{m}H^2(K_i,\wh{T}_{E_i/K_i})}.     
\end{tikzcd}   
\end{equation}
\end{footnotesize}

Consider the lower exact sequence of (\ref{eq:five lemma (ab)}). 
Set 
\begin{equation}\label{eq:6.5}
   \Sel(\wh S_{K,K'}):=\left\{ x\in H^1(k, \wh S_{K,K'}): \delta(x)\in \Sha^2(k,\wh T_{K/k}) \right\}. 
\end{equation}
Since $\im{\rho}\subseteq \Sel(\wh S_{K,K'})$, we have \[ \Ker(\delta\mid_{\Sel(\wh S_{K,K'})})=\Ker\delta\cap \Sel(\wh S_{K,K'})=\im{\rho}\cap \Sel(\wh S_{K,K'})=\im{\rho}. \] 
On the other hand, we have
\[ \Ker(\iota^2\mid_{\Sha^2(k,\wh T_{K/k})})=\Ker{\iota^2}\cap\Sha^2(k,\wh T_{K/k})=\im{\delta}\cap \Sha^2(k,\wh T_{K/k})=\delta(\Sel(\wh S_{K,K'})). \] Therefore, we obtain an exact sequence:
\begin{footnotesize}
\begin{center}
\begin{tikzcd}
H^1(k,\wh{T}_{K/k}) \arrow[r, "\iota"] & {\bigoplus\limits_{i=1}^{m}H^1(K_i,\wh{T}_{E_i/K_i})} \arrow[r, "\rho"] & \Sel(\wh{S}_{K,K'}) \arrow[r, "\delta"] & \Sha^2(k,\wh{T}_{K/k}) \arrow[r, "\iota^2"] & {\bigoplus\limits_{i=1}^{m}\Sha^2(K_i,\wh{T}_{E_i/K_i})}.
\end{tikzcd}
\end{center}
\end{footnotesize}
Similarly, we define $\Sel(\wh S_{K(p),K'})$ in the same way and have the exact sequence 
\begin{footnotesize}
\begin{center}
\begin{tikzcd}
H^1(k,\wh{T}_{K(p)/k}) \arrow[r, "\iota(p)"]  & {\bigoplus\limits_{i=1}^{m}H^1(K_i,\wh{T}_{E_i(p)/K_i})} \arrow[r, "\rho(p)"] & \Sel(\wh{S}_{K(p),K'}) \arrow[r, "\delta(p)"] & \Sha^2(k,\wh{T}_{K(p)/k}) \arrow[r, "\iota(p)^2"] & {\bigoplus\limits_{i=1}^{m}\Sha^2(K_i,\wh{T}_{E_i(p)/K_i})}.
\end{tikzcd}
\end{center}
\end{footnotesize}
 By definition, $\wh N_{E/E(p)}(\Sel(\wh S_{K(p),K'}))\subset \Sel(\wh S_{K,K'})$, so we combine the two exact sequences above into the following commutative diagram:

\begin{footnotesize}
\begin{equation}\label{eq: Sel(S)_five lemma}
\begin{tikzcd}
H^1(k,\wh{T}_{K(p)/k}) \arrow[r, "\iota(p)"] \arrow[d, "\wh N_{K/K(p)}"]& {\bigoplus\limits_{i=1}^{m}H^1(K_i,\wh{T}_{E_i(p)/K_i})} \arrow[r, "\rho(p)"] \arrow[d, "\wh N_{E/E(p)}"] & \Sel(\wh{S}_{K(p),K'}) \arrow[r, "\delta(p)"] \arrow[d, "\wh N_{E/E(p),S}"] & \Sha^2(k,\wh{T}_{K(p)/k}) \arrow[r, "\iota(p)^2"] \arrow[d, "\wh N_{K/K(p)}^2"] & {\bigoplus\limits_{i=1}^{m}\Sha^2(K_i,\wh{T}_{E_i(p)/K_i})} \arrow[d, "\wh N_{E/E(p)}^2"]\\
H^1(k,\wh{T}_{K/k}) \arrow[r, "\iota"] & {\bigoplus\limits_{i=1}^{m}H^1(K_i,\wh{T}_{E_i/K_i})} \arrow[r, "\rho"] & \Sel(\wh{S}_{K,K'}) \arrow[r, "\delta"] & \Sha^2(k,\wh{T}_{K/k}) \arrow[r, "\iota^2"] & {\bigoplus\limits_{i=1}^{m}\Sha^2(K_i,\wh{T}_{E_i/K_i})}.
\end{tikzcd}
\end{equation}
\end{footnotesize}

Since the abelian groups in the upper row of \eqref{eq: Sel(S)_five lemma} are all $p$-groups, we have reached the following commutative diagram: 
\begin{footnotesize}
\begin{equation}\label{eq: Sel(S)_five lemma2}
\begin{tikzcd}
H^1(k,\wh{T}_{K(p)/k}) \arrow[r, "\iota(p)"] \arrow[d, "\wh N_{K/K(p)}"]& {\bigoplus\limits_{i=1}^{m}H^1(K_i,\wh{T}_{E_i(p)/K_i})} \arrow[r, "\rho(p)"] \arrow[d, "\wh N_{E/E(p)}"] & \Sel(\wh{S}_{K(p),K'}) \arrow[r, "\delta(p)"] \arrow[d, "\wh N_{E/E(p),S}"] & \Sha^2(k,\wh{T}_{K(p)/k}) \arrow[r, "\iota(p)^2"] \arrow[d, "\wh N_{K/K(p)}^2"] & {\bigoplus\limits_{i=1}^{m}\Sha^2(K_i,\wh{T}_{E_i(p)/K_i})} \arrow[d, "\wh N_{E/E(p)}^2"]\\
H^1(k,\wh{T}_{K/k})(p) \arrow[r, "\iota"] & {\bigoplus\limits_{i=1}^{m}H^1(K_i,\wh{T}_{E_i/K_i})(p)} \arrow[r, "\rho"] & \Sel(\wh{S}_{K,K'})(p) \arrow[r, "\delta"] & \Sha^2(k,\wh{T}_{K/k})(p) \arrow[r, "\iota^2"] & {\bigoplus\limits_{i=1}^{m}\Sha^2(K_i,\wh{T}_{E_i/K_i})(p)}.
\end{tikzcd}
\end{equation}
\end{footnotesize}

By Lemma~\ref{lem: decomp_H^1(T) Norm} and Corollary~\ref{wh N^2}, the maps $\wh N_{K/K(p)}$ and $\wh N_{K/K(p)}^2$ are isomorphisms. The same reason also implies that the maps $\wh N_{E/E(p)}$ and $\wh N_{E/E(p)}^2$ are also isomorphisms. By the five lemma, the norm map
\[ \wh N_{E/E(p),S}: \Sel(\wh{S}_{K(p),K'}) \isoto \Sel(\wh{S}_{K,K'})(p) \]
is an isomorphism.

It is obvious that $\Sha{^1}(k,\wh S_{K,K'})(p)\subseteq \Sel(\wh S_{K,K'})(p)$ and $\Sha{^1}(k,\wh S_{K(p),K'})\subseteq \Sel(\wh S_{K(p),K'})$. Thus, 
$\Sha{^1}(k,\wh S_{K,K'})(p)$ (resp.~$\Sha{^1}(k,\wh S_{K(p),K'})$) is the kernel of the global-to-local map in $\Sel(\wh S_{K,K'})(p)$ (resp.~in $\Sel(\wh S_{K(p),K'})$). 

To finish, we consider the local version of the commutative diagram \eqref{eq: Sel(S)_five lemma2}:
\begin{footnotesize}
\begin{equation}\label{eq:local_five_lemma}
\begin{tikzcd}
\prod\limits_{v} H^1(k_v,\wh{T}_{K(p)/k}) \arrow[r, "\iota(p)"] \arrow[d, "\wh N_{K/K(p),\A}"]& {\bigoplus\limits_{i=1}^{m} \prod\limits_{w} H^1(K_{i,w},\wh{T}_{E_i(p)/K_i})} \arrow[r, "\rho(p)"] \arrow[d, "\wh N_{E/E(p),\A}"] & \rho(p)\left( {\bigoplus\limits_{i=1}^{m} \prod\limits_{w} H^1(K_{i,w},\wh{T}_{E_i(p)/K_i})} \right ) 
\arrow[r, ""] \arrow[d, "\wh N_{E/E(p),S,\A}"] & 0  \\
\prod\limits_v H^1(k_v,\wh{T}_{K/k})(p) \arrow[r, "\iota"] & {\bigoplus\limits_{i=1}^{m} \prod\limits_{w} H^1(K_{i,w},\wh{T}_{E_i/K_i})(p)} \arrow[r, "\rho"] & \rho \left (  {\bigoplus\limits_{i=1}^{m} \prod\limits_{w} H^1(K_{i,w},\wh{T}_{E_i/K_i})(p)}  \right )
\arrow[r] & 0.
\end{tikzcd}
\end{equation}
\end{footnotesize}

Since $\wh N_{K/K(p),\A}$ and $\wh N_{E/E(p),\A}$ are isomorphic, so is the map $\wh N_{E/E(p),S,\A}$. Finally, we obtain the following commutative diagram:
\begin{center}
\begin{tikzcd}
\Sel(\wh S_{K(p),K'}) \arrow[r, "\wh N_{E/E(p),S}", "\sim"'] \arrow[d]
& \Sel(k, \wh S_{K,K'})(p) \arrow[d] \\
\rho(p)\left( {\bigoplus\limits_{i=1}^{m} \prod\limits_{w} H^1(K_{i,w},\wh{T}_{E_i(p)/K_i})} \right ) 
\arrow[r, "\wh N_{E/E(p),S,\A}", "\sim" '] & 
\rho \left (  {\bigoplus\limits_{i=1}^{m} \prod\limits_{w} H^1(K_{i,w},\wh{T}_{E_i/K_i})(p)}  \right ),
\end{tikzcd}    
\end{center}
where the vertical maps are the global-to-local maps. It follows that 
$\wh N_{E/E(p)}: \Sha{^1}(k,\wh S_{K(p),K'}) \isoto \Sha{^1}(k,\wh S_{K,K'})(p)$ is an isomorphism. \qed
\end{proof}

We apply Theorem~\ref{thm:G(p)xH} to the case where $K/k$ is a nilponent extension and get the following result. Note that any finite nilponent group is the product of its Sylow $p$-subgroups for all $p$.

\begin{cor}\label{cor:nilpotent}
Let $L=K\times K'$ be a finite \'etale $k$-algebra. If $K$ is a nilpotent field extension over $k$ of degree $d_0$, then there is a natural isomorphism 
\[ \Sha(L/k)\isoto \bigoplus_{p|d_0} \Sha(L[p]/k), \] 
where $L[p]:=K(p)\times K'$.
\end{cor}


\begin{prop}\label{prop:nitlponent}
Let $L=K\times \prod_{i=1}^m K_i$ be an \'etale $k$-algebra. Assume that $K$ and $K_i$ are all nilpotent, and let $L(p)=K(p)\times \prod_i K_i(p)$. Then 
\[ \Sha(L/k)=\bigoplus_{p|d} \Sha(L(p)/k), \]
where $d=\gcd([K_0:k], \dots, [K_m:k])$.
\end{prop}

\begin{proof}
It suffices to show that $\Sha(L/k)(p)=\Sha(L(p)/k)$. Using Theorem~\ref{thm:G(p)xH} and induction, we have 
\begin{equation*}
    \begin{split}
    \Sha(L/k)(p)&=\Sha(K(p)\times K'/k)(p)=\Sha(K_1\times K(p)\times K_2 \times \cdots\times K_m /k)(p)\\
              &=\Sha(K_1(p)\times K(p)\times K_2 \times \cdots\times K_m/k)(p) \\
              &=\Sha(K_2\times K(p)\times K_1(p) \times K_3 \times \cdots\times K_m/k)(p) \\
              & \ \ \vdots \\
              &=\Sha(K(p)\times \prod_i K_i(p)/k)\\
              &=\Sha(L(p)/k).  \ \text{\qed}
\end{split}
\end{equation*}

\end{proof}

\section*{Acknowledgments}
The authors thank E. Bayer-Fluckiger, T.-Y. Lee, and R. Parimala for their inspiring paper \cite{BLP19}, which has been providing the guide for working on this topic.  
This article is an extension of the results presented in the master thesis of first named author at National Tsing-Hua University. 
The first and third named authors are partially supported by the the National Science and Technology Council grants 109-2115-M-001-002-MY3, 112-2115-M-001-010, and the Academia Sinica Investigator Grant AS-IA-112-M01. The second author is supported by JSPS Research Followship for Young Scientists and KAKENHI Grant Number 22KJ0041. They thank Valentijn Karemaker for a careful reading of an earlier draft and her helpful comments. 
The referee's careful reading and invaluable comments on the manuscript are deeply acknowledged.

\bibliographystyle{plain}

\begin{thebibliography}{10}

\bibitem{BLP19}
E.~Bayer-Fluckiger, T.-Y. Lee, and R.~Parimala.
\newblock Hasse principles for multinorm equations.
\newblock {\em Adv. Math.}, 356:106818, 35, 2019.

\bibitem{demarche-wei}
Cyril Demarche and Dasheng Wei.
\newblock Hasse principle and weak approximation for multinorm equations.
\newblock {\em Israel J. Math.}, 202(1):275--293, 2014.

\bibitem{HHLY}
Jun-Hao Huang, Fan-Yun Hung, Pei-Xin Liang, and Chia-Fu Yu.
\newblock Computing {T}ate-{S}hafarevich groups of multinorm one tori of {K}ummer type, {\em Taiwanese J. Math.} 28(4), 671–
694, 2024.



\bibitem{hung-yu} Fan-Yun Hung, and Chia-Fu Yu, Class numbers of multinorm one tori. {\em J. Number Theory}
258, 94–-121, 2024.


\bibitem{hurlimann}
W.~H\"{u}rlimann.
\newblock On algebraic tori of norm type.
\newblock {\em Comment. Math. Helv.}, 59(4):539--549, 1984.


\bibitem{lee:jpaa2022}
T-Y Lee.
\newblock The {T}ate-{S}hafarevich groups of multinorm-one tori.
\newblock {\em Journal of Pure and Applied Algebra}, 226(7):106906, 2022.

\bibitem{LOYY}
Pei-Xin Liang, Yasuhiro Oki, Hsin-Yi Yang, and Chia-Fu Yu. 
\newblock On {T}amagawa numbers of {CM} tori. Appendix by Jianing Li and Chia-Fu Yu.
\newblock {\it Algebra Number Theory} Vol. 18 (2024), No. 3, 583--629.





\bibitem{platonov-rapinchuk}
Vladimir Platonov and Andrei Rapinchuk.
\newblock {\em Algebraic groups and number theory}, volume 139 of {\em Pure and
  Applied Mathematics}.
\newblock Academic Press, Inc., Boston, MA, 1994.
\newblock Translated from the 1991 Russian original by Rachel Rowen.

\bibitem{pollio}
Timothy~P. Pollio.
\newblock On the multinorm principle for finite abelian extensions.
\newblock {\em Pure Appl. Math. Q.}, 10(3):547--566, 2014.

\bibitem{pollio-rapinchuk}
Timothy~P. Pollio and Andrei~S. Rapinchuk.
\newblock The multinorm principle for linearly disjoint {G}alois extensions.
\newblock {\em J. Number Theory}, 133(2):802--821, 2013.

\bibitem{prasad-rapinchuk}
Gopal Prasad and Andrei~S. Rapinchuk.
\newblock Local-global principles for embedding of fields with involution into
  simple algebras with involution.
\newblock {\em Comment. Math. Helv.}, 85(3):583--645, 2010.

\bibitem{Serre_local}
Jean-Pierre Serre.
\newblock {\em Local fields}, volume~67 of {\em Graduate Texts in Mathematics}.
\newblock Springer-Verlag, New York, 1979.
\newblock Translated from the French by Marvin Jay Greenberg.

\end{thebibliography}

\def\cprime{$'$}

\end{document}